\documentclass[11pt,oneside]{amsart}
\usepackage{amscd,amsmath,amssymb,amsfonts}
\usepackage[all]{xy}
\usepackage{hyperref}
\usepackage{url}
\usepackage{stmaryrd}

\usepackage[a4paper, total={6in, 8in}]{geometry}

\usepackage{xcolor}

\usepackage{upgreek}

\theoremstyle{plain}
\newtheorem{thm}{Theorem}
\newtheorem{lem}[thm]{Lemma}
\newtheorem{cor}[thm]{Corollary}
\newtheorem{prop}[thm]{Proposition}

\theoremstyle{definition}

\newtheorem{claim}[thm]{Claim}
\newtheorem{rmk}[thm]{Remark}

\numberwithin{thm}{section}
\numberwithin{equation}{section}

\newcommand{\ga}[2]{\begin{gather}\label{#1}#2 \end{gather}}
\newcommand{\gr}{{\rm gr}}

\DeclareMathOperator{\im}{{\rm im}}
\newcommand{\Spec}{{\rm Spec \,}}

\newcommand{\sF}{{\mathsf F}}

\newcommand{\sO}{{\mathcal O}}

\newcommand{\C}{{\mathbb C}}

\renewcommand{\P}{{\mathbb P}}
\newcommand{\Q}{{\mathbb Q}}
\newcommand{\R}{{\mathbb R}}

\newcommand{\Z}{{\mathbb Z}}

\newcommand{\GL}{{\mathrm{GL}}}

\begin{document}

\title{A non-abelian version of Deligne's Fixed Part Theorem}
\author{H\'el\`ene Esnault \and Moritz Kerz}
\address{Freie Universit\"at Berlin, Berlin,  Germany;
{Harvard} University,  Cambridge, USA
}
\email{esnault@math.fu-berlin.de }
\address{   Fakult\"at f\"ur Mathematik \\
Universit\"at Regensburg \\
93040 Regensburg, Germany}
\email{moritz.kerz@mathematik.uni-regensburg.de}
\thanks{ The second author by the SFB 1085 Higher Invariants, Universit\"at Regensburg}

\begin{abstract}
  We formulate and prove a non-abelian analog of Deligne's Fixed Part theorem on Hodge
  classes, revisiting previous work of Jost--Zuo, Katzarkov--Pantev and Landesman--Litt.
  To this aim we  study algebraically isomonodromic extensions of local systems, and we
  relate them to variations of Hodge structures, for example we show that
  the Mumford-Tate group at a generic point stays constant in an algebraically
  isomonodromic extension of a variation of Hodge  structure.
\end{abstract}

\maketitle

\medskip

\section{Introduction}\label{sec:intro}

 On a   quasi-projective complex manifold, a rational Betti class  is called Hodge if it
 is integral and lies as a de Rham class in the right   levels of the Hodge and the weight filtrations.
   Let $f \colon X\to S$ be a projective morphism between quasi-projective complex
   manifolds.  Deligne’s Fixed Part Theorem asserts that a Hodge class on a fibre $X_s$  has
   finite  orbit under the fundamental group  of $S$ based at $s$  if and only if it
   extends to a Hodge class on $X$ after base change to a finite  \'etale cover of $S$.
   Here we have to make the technical assumption that $S$ topologically has the properties of
   Artin's elementary neighborhoods, see Theorem~\ref{thm:del_fixpart} for a precise formulation.

Replacing above the Hodge class with a rational polarized  variation of Hodge structure
which admits an integral structure, Deligne’s Fixed Part Theorem  has a perfect non-abelian
analog, which is the topic of our note.

\smallskip

We  introduce some notation in order to give a  precise formulation  of the main theorem.
On a complex manifold $X$,  there is  the notion of a {\it polarized variation of $K$-Hodge structure}  ($K \text{-}{\sf PVHS}$), for
$K\in \{ \mathbb Q, \mathbb R, \mathbb C \}$, see  Subsection~\ref{subsec:basichodgevar}. It consists of a triple
 $(\mathbb L, \sF , Q)$, where $\mathbb L$ is a $K$-local system on $X$, $\sF$ is a finite
 filtration by holomorphic subbundles of   $\mathbb L\otimes_K \sO_X$ and $Q\colon \mathbb
 L \otimes \overline{\mathbb L}\to K$ is a sesquilinear perfect pairing of local systems. This triple  is
 asked to satisfy the Hodge-Riemann relations and Griffiths transversality.

If  $R\subset K$ is a subring we denote by $(R)K \text{-}{\sf PVHS}$ those polarized
variations of $K$-Hodge structure $(\mathbb L, \sF , Q)$ such that the local system
$\mathbb L$ can be defined over $R$ in the classical sense, that is that its monodromy representation can be defined
over $R$.

 Let  $f: X\to S$ be an algebraic morphism between quasi-projective complex
  manifolds.  We say that $f$ is a {\it good algebraic morphism} if there is a smooth extension $\overline f:    \overline X \to S$ of $f$ such that $\overline X \supset X$ is quasi-projective smooth and $D=\overline X\setminus X$ is a relative normal crossings divisor.   
 We prove our main results assuming that $f$ satisfies one of the
following conditions, and we assume for the rest of the introduction that
this is the case.  \\
{\bf Ass:}
$f: X\to S$ is a good algebraic morphism
and at least one of the following properties holds for $f$: 
\begin{itemize}
\item[I)] $S$ is an Artin  $K\pi1$, i.e.\ its fundamental group is a successive extension
  of finitely generated free groups and $\pi_i(S)=0$ for $i>1$,
\item[II)]  $X_s$ is a hyperbolic Riemann surface,
\item[III)] $f$ has a continuous section $S\to X$.
\end{itemize}
For the precise notion of an Artin  $K\pi1$ see Subsection~\ref{sec:Artin}. A hyperbolic Riemann
surface $Y$ for us is the complement of finitely many points in a compact connected Riemann surface
$\overline Y$ such that the
 Euler characteristic $\chi(Y)$ is negative.

\medskip

Our main theorem is the following non-abelian Fixed Part Theorem, see
Section~\ref{sec:proof_main} for the proof.

\begin{thm} \label{thm:main}   Let $f: X\to S$ be a good
  algebraic morphism  fulfilling {\bf Ass}, $s$   be a point in $S$.
Let  $(\mathbb L_s, \sF_s, Q_s)$  be a $(\Z) K \text{-}{\sf PVHS} $ on the fiber $X_s$, where $K\in
\{\Q,\R,\C\}$.
Then the following conditions are equivalent:
\begin{itemize}
\item[1)]   the orbit $\pi_1(S,s) \cdot [\mathbb L_s]$ in the set of isomorphism classes
  of $K$-local systems on $X_s$ is finite;
\item[2)]   $\mathbb L_s$ essentially extends to a $K$-local system  $\mathbb L$ on $X$;
 \item[3)] $(\mathbb L_s,\sF_s,Q_s)$ essentially extends to  a  $(\Z)K \text{-}{\sf PVHS} $
   $(\mathbb L, \sF, Q)$ on $X$;
 \item[4)]  $(\mathbb L_s, \sF_s, Q_s)$  extends to a $K \text{-}{\sf PVHS} $
on $X_\Delta=f^{-1}(\Delta)$ for some contractible open  neighborhood $\Delta\subset S$ of
$s$.
 \item[5)]  the filtration $\sF_s$  extends to a (relative) Griffiths
  transversal filtration by  subbundles
  $\hat \sF$ of the formal isomonodromic Deligne extension $(\hat{\mathcal E},\hat \nabla\colon
  \hat{\mathcal E} \to \hat{\Omega}^1_{\hat{\overline  X}/ \hat S}(\log D)(\hat{\mathcal
    E}))$.
\end{itemize}
\end{thm}

Here ``essentially extends'' means that it extends after base change with respect to a
finite, \'etale, surjective covering of $S$, see Subsection~\ref{ss:IFP} for a precise
formulation.
  In 5)  $\hat S$ is the formal scheme of $S$ along $s$ and
  $\hat{ \overline X}$ the formal scheme of $\overline X$ along $ \overline X_s$. The formal isomonodromic Deligne extension
  $(\hat{\mathcal E},\hat \nabla)$ is the formal completion of the log flat bundle on
  $\overline X_\Delta$
  associated by Deligne~\cite[Prop.~5.4, Rmk.~5.5]{Del70} to the local system on $X_\Delta$ which canonically
  extends $\mathbb L_s$. Recall that the filtration $\sF_s$ canonically extends to a filtration by
  subbundles of the Deligne extension over $\overline X_s$ by Schmid's nilpotent orbit
  theorem, see \cite[(2.1)]{CK89}, so part 5) can be
  interpreted as a formal deformation problem for the latter filtration of bundles from
  $\overline X_s$ to $\hat{\overline  X}$.

The idea of the proof is to show the implication 1) $\Rightarrow $ 3) by constructing an
algebraically isomonodromic extension of $\mathbb L_s$ to $X$, see Theorem~\ref{introthm2}, and to use
T.\ Mochizuki's non-abelian Hodge correspondence, see Subsection~\ref{ss:mochi}, which enables one to relate the action of the $\C^\times$ flow on the Higgs bundle associated to $\mathbb L_s$  to the one on an algebraically  isomonodromic extension $\mathbb L$.

In order to
show the implication 4) $\Rightarrow$ 1) we use Simpson's non-abelian Hodge loci. In fact,
he considers only the case $f$ projective, and we generalize  in Appendix~\ref{app:nonabhodge}  his theory to the case of a
 {\it good holomorphic map} $f: X\to S$, that is  $X$ and $S$ are  complex manifolds, $f$ is submersive and  admits
a relative compactification $\overline f: \overline X\to S$ such that  $\overline f$ is a projective morphism and $D=\bar X \setminus X$ is a relative normal crossings divisor.
 The analogous implication 5) $\Rightarrow$ 1)  was
  suggested to us by D.~Litt, see also~\cite[Cor.~4.4.3]{Lit24}.

\begin{rmk}\label{intro:rmk1}
Theorem~\ref{thm:main} has a well-known ``Lefschetz hyperplane version'', which is deduced from
T.~Mochizuki's work in the same way.
Let us assume $\overline X$ is a projective compactification of $X$ with  a  simple normal crossings divisor  $D= \overline X\setminus X$ at infinity.
Consider a  hyperplane section  $ \overline{Y}\subset \overline{X}$ for a projective
embedding $\overline X\hookrightarrow \mathbb P^N_{\C}$ which is transversal  to
the stratification induced by $D$, i.e.\ such that the intersection of $\overline Y$ and
of irreducible components of $D$ is smooth of the right dimension.
Assume that $\dim X\ge 2$ and recall that then the map $\pi_1(Y)\to \pi_1(X)$
is surjective~\cite[II.1.1]{GM88},  where $Y=\overline{Y}\cap X$. Let $\mathbb L$ be a $K$-local system on $X$ such that $\mathbb L|_Y$
underlies  a $(\Z ) K\text{-}\mathsf{PVHS}$ $(\mathbb L|_Y, \mathsf F_Y, Q_Y)$. Then the
filtration $\mathsf F_Y$ and the polarization $Q_Y$ uniquely extend to a   $(\Z )
K\text{-}\mathsf{PVHS}$ $(\mathbb L, \mathsf F, Q)$.
\end{rmk}

\begin{rmk}\label{rmk:introMT}
 For $K=\Q$ one can choose the extension in 3)   such that
   the Hodge-generic locus of $(\mathbb L,\sF,Q)$ meets $X_s$, or stated differently such that the  Mumford-Tate groups of
  $(\mathbb L_s,\sF_s,Q_s)$ and of  $(\mathbb L,\sF,Q)$ are the same, see Section~\ref{sec:proof_main}.

  The same argument implies that in Remark~\ref{intro:rmk1} the Hodge-generic locus of
  $(\mathbb L, \mathsf F, Q)$ meets $Y$.
\end{rmk}

As mentioned above our new method is to study {\it algebraically   isomonodromic extensions} and to relate them to  rigidity properties of the Hodge data on local systems.
The notion of an  isomonodromic extension of a bundle with a flat connection  on a fiber  $X_s$ of a
good holomorphic map $f \colon X\to S$  to a relative flat bundle over a neighborhood
of $s\in  S$
 is well-established in the theory of non-linear differential equations;  its connection to Hodge
 theory has been studied in particular in~\cite{LL22}.

  It turns out that the right notion  for our global study of  local systems is not the
  discrete monodromy group itself, but the  identity component of its Zariski
  closure.
For a $K$-local system $\mathbb L$ on $X$ and $x\in X$ we denote by $\mathrm{Mo}(\mathbb
L,x)$  the Zariski closure in $\mathrm{GL}(\mathbb L_x)$ of the image of the monodromy
representation
$
\pi_1(X,x)\to \mathrm{GL}(\mathbb L_x)$ and by   $\mathrm{Mo}(\mathbb
L,x)^0$ its identity component. One calls  $\mathrm{Mo}(\mathbb
L,x)^0$ the {\it algebraic monodromy group}.

  We say that an extension of a
  local system to a bigger manifold is {\it algebraically isomonodromic} if the two algebraic
  monodromy groups are the same.
 We first prove by group theoretic means,  developed in Sections~\ref{sec:groupth}
 and~\ref{sec:extreps} and Appendix~\ref{app:extreps},  the following theorem, which is
 shown in Section~\ref{sec:proof_main}.  We fix a field $K$ of characteristic zero.

\begin{thm}\label{introthm2} Let $f \colon X\to S$ be a
    good algebraic morphism fulfilling {\bf Ass}.
 Let $\mathbb L_s$ be a $K$-local system on $X_s$ with semisimple algebraic monodromy
group.
Then the  following conditions are equivalent:
\begin{itemize}
\item[a)] the orbit $\pi_1(S,s) \cdot [\mathbb L_s]$ in the set of isomorphism classes of
  $K$-local systems on $X_s$ is finite;
\item[b)]  $ \mathbb L_s$ essentially extends to a local system
   $\mathbb L$ on $X$;
\item[c)] $\mathbb L_s$ essentially extends to an algebraically isomonodromic local system $\mathbb L$ relative to $X_s$.

\end{itemize}
Moreover, the following holds:
\begin{itemize}
\item[i)]  if  $\mathbb L$  as in c) exists it is essentially unique;
\item[ii)]  if $K$ is a number field and $\mathbb L_s$
has an integral lattice $\mathbb L_{s, \sO_K}$, then $\mathbb L$  as in c) essentially has an integral  lattice $\mathbb L_{ \sO_K}$;
\item[ iii)] if $\mathbb L_s$ is absolutely simple, then  an extension $\mathbb L$ to $X$
which  has finite determinant is algebraically  isomonodromic relative to $X_s$.
\end{itemize}
\end{thm}

Recall that if $\mathbb L_s $ underlies a  $(\Z) \C \text{-}{\sf PVHS} $ then  by~\cite[Sec.~4.2]{Del71}  its
algebraic monodromy group is semisimple, see Theorem~\ref{thm:ss}(v). We do not know
whether the lattice $\mathbb L_{\sO_K}$   in ii) can be chosen as an extension of  $\mathbb L_{s,\sO_K}$.

It is a classical observation due to Clifford that in the setting of
Theorem~\ref{introthm2} one can canonically extend the representation of the fundamental group
associated to an absolutely simple local system $\mathbb L_s$ to a projective representation of
$\pi_1(X)$. The problem is to lift the latter to a proper representation. Our
observation is that this lifting process, which a priori has an obstruction in
$H^2(\Gamma,K^\times)$, is more manageable in the world of pro-finite groups.
Our proof of Theorem~\ref{introthm2} in fact relies on passing to continuous representations over
the $\ell$-adic numbers of the profinite completion of fundamental groups. In the
profinite world the continuous $H^2$-obstruction can be killed by passing to an open
subgroup.

\smallskip

We can formulate
an application of Theorem~\ref{introthm2} to representations of mapping class
groups which refines~\cite[Prop.~2.3.4]{LLcan}, see Section~\ref{sec:proof_main} for
a proof.

\begin{cor}\label{cor:introcanrepsurfg}
Let  $S$ be a finite \'etale covering of the moduli stack of hyperbolic punctured compact Riemann surfaces
with universal curve $f\colon X\to S$. Fix $s\in S$. Then a semisimple local system  $\mathbb
L_s$ on
$X_s$   such that the orbit $\pi_1(S,s) \cdot
[\mathbb L_s]$ is finite, essentially extends to $X$.
\end{cor}

\smallskip

Many of the results of our note are known to the experts in one form or another. In
particular we refer to the previous work by Jost--Zuo~\cite{JZ01} and
Katzarkov--Pantev~\cite{KP02}. Our motivation to write down a complete account of their
theory stems from our ongoing work on the arithmetic of $\ell$-adic local systems over
$p$-adic local fields. In that work we for example study an arithmetic variant of
non-abelian Hodge loci for
$\ell$-adic local systems.

\medskip

{\it Acknowledgment}. It is a pleasure to thank Alexander Petrov for numerous
enlightening discussions and for sharing his insights into isomonodromic deformations of
Hodge structures.  In particular, he  alerted us on  a gap in a previous version of Appendix~\ref{app:nonabhodge}.
We thank Claude Sabbah for kindly helping us with
references on variations of Hodge structure. We thank Daniel Litt for various helpful discussions on local
systems  in particular on the condition 5) in
  Theorem~\ref{thm:main}.
We are grateful to Takuro Mochizuki for explaining to us what is known and what can be expected about families of pluri-harmonic metrics.  We thank Xinwen Zhu for an enlightening discussion on the analogy between  Theorem~\ref{thm:main} and his  own theorem on the preservation of the de Rham property of a $p$-adic local system, see  \cite{LZ17}. 
It is a pleasure to acknowledge the influence of~\cite{LLcan} on our work.
Finally, we thank the referee for a friendly, precise and helpful report.

\section{Profinite completion of groups and Artin $K\pi1$ spaces}  \label{sec:groupth}

In this section we recall some classical properties of profinite group completion in the
context of fundamental groups of quasi-projective complex manifolds.

\subsection{Group theory}

We recall after Anderson~\cite{And74}  some criteria  under which
   a finitely generated  normal subgroup
$H\subset G$  of a finitely generated  abstract group $G$  induces an injective homomorphism $\widehat H\to \widehat G$ after profinite completion.

\smallskip

Note first that
\begin{quote}  $\widehat H\to \widehat G$ is injective if and only
if for any  normal  subgroup of finite index $H'\subset H$ there exists a  subgroup of finite index
$G'\subset G$ such that $G'\cap H \subset H'$.
\end{quote}
 We remark that we could equally request the existence of a normal subgroup of finite index $G' \subset G$ as any subgroup of finite index contains a normal such.
We  denote by $\Gamma$ the quotient group $G/H$.

We say that $\Gamma$ is
  {\it a successive
extension of finitely generated free groups}
 if there is a filtration
\[ (\star) \ \ \ \
\{1\}=\Gamma_0\subset   \Gamma_1\subset \Gamma_2 \subset \cdots \subset \Gamma_s = \Gamma
\]
by finitely generated subgroups $\Gamma_i$ such that $\Gamma_{i-1}$ is normal in $
\Gamma_{i}$ and
such that $\Gamma_i/\Gamma_{i-1}$ is a free group for all  $1\le i \le s$.
Note that if $\Gamma$ is a successive
extension of finitely generated free groups and $\Gamma'\subset \Gamma$ is a subgroup of
finite index,  then $\Gamma'$ is itself a successive
extension of finitely generated free groups.

\begin{lem}\label{lem:finitekill}
If $\Gamma$ is a successive
extension of finitely generated free groups,  $N>0$ and $i>0$ are integers, there exists a subgroup of
finite index $\Gamma'$ such that the restriction map
\[
H^i(\Gamma,\Z/N\Z) \to H^i(\Gamma',\Z/N\Z)
\]
vanishes.
\end{lem}

\begin{proof}
  This is Exercise 2)(d) in \cite[I.2.6]{Ser94}.
\end{proof}

\begin{prop}\label{prop:ander1}
If the surjective homomorphism $G\to \Gamma$ has a section or   if $\Gamma$ is a successive
extension of finitely generated free groups,  then $\widehat H\to\widehat G$ is injective.
\end{prop}

\begin{proof}\label{prop:ander2}
 The first part is  \cite[Prop.\ 4]{And74}. For the second part let us assume that
 $\Gamma$ is an extension of $s>0$ finitely generated free groups as above.
We first assume that $s=1$ that is that  $\Gamma$ is free. Then the surjection $G\to \Gamma$
has a section and we can conclude as before.  For $s\ge 2$, we define $G_{s-1}$ as the inverse image in $G$ of $\Gamma_{s-1}$.
By induction on $s$ we can assume that $\hat H\to \hat G_{s-1}$ is injective. By the case $s=1$,  $\hat G_{s-1}\to \hat G_s$ is injective. Thus,
the composite $\hat H\to \hat G_{s-1} \to \hat G_s$ is injective as well.
\end{proof}

The next proposition is  \cite[Prop.\ 3]{And74}.
\begin{prop}\label{prop:ander3}
  If $\widehat H$ has trivial center then $\widehat H\to\widehat G$ is injective.
\end{prop}
\subsection{Geometric examples} \label{sec:Artin}

Let $X $ and $S$ be connected complex manifolds and let $f\colon X\to S$ be a
  good holomorphic
map  as defined in Section~\ref{sec:intro}. Fix a point $x\in X$ and set
$s=f(x)$.
 By Thom's first isotopy lemma~\cite[Thm.~I.1.5]{GM88} or Ehresmann's theorem in the proper case, $f$  is a
 topological  fiber bundle.
We obtain  an exact sequence of homotopy groups
\begin{equation}\label{eq:longexhom}
\ldots \to  \pi_2(X,x) \to  \pi_2(S,x)   \to \pi_1(X_s,x) \to \pi_1(X,x)\to \pi_1(S,s) \to 1
\end{equation}

We say that $S$ is {\it an Artin} $K\pi 1$ if
$\pi_i(S)=0$ for $i>1$ and $\pi_1(S)$ is a  successive
extension of finitely generated free groups. Recall that Artin showed
\cite[4.6]{SGA4.3}
 that if
$S$ underlies a smooth algebraic variety,  Artin $K\pi1$'s  form a base of the Zariski topology.

\begin{prop}\label{prop:artKpiki}
  Let $\mathbb L$ be a local system on $S$ with finite fibers.
If $S$ is an Artin $K\pi 1$ and $i>0$ is an integer, there exists a finite \'etale cover  $S'\to
S$ such that the pull-back map
\[
H^i(S,\mathbb L ) \to H^i(S' ,\mathbb L)
\]
vanishes.
\end{prop}

\begin{proof}

  As  $S$ is an Artin $K\pi 1$  the morphism
  $
H^i(\pi_1(S,s) ,\mathbb L_s)\xrightarrow \sim H^i(S,\mathbb L)
  $
  stemming from the Hochschild-Serre spectral sequence is an isomorphism.  We apply
 Lemma~\ref{lem:finitekill}  to the left term.
 \end{proof}

\begin{prop} \label{prop:Kpi1}
If $S$ is an Artin $K\pi 1$,  then $\pi_1(X_s,x) \to \pi_1(X,x)$ is injective and remains
injective after profinite completion.
\end{prop}

\begin{proof}
By definition $\pi_2(S,s)=0$ so $\pi_1(X_s,x)\to \pi_1(X,x)$ is injective in view of~\eqref{eq:longexhom}.
 We apply  Proposition~\ref{prop:ander1}.

 \end{proof}

\begin{prop} \label{prop:hyp}
If $X_s$ is a hyperbolic Riemann surface, the  map $\pi_1(X_s,x) \to \pi_1(X,x)$ is injective and remains
injective after profinite completion.
\end{prop}

\begin{proof}
By  \cite[Prop.\ 18]{And74}, the centers of $\pi_1(X_s,x)$ and of $\widehat{\pi_1(X_s,x)}$ are trivial.  Thus, by Proposition~\ref{prop:ander3}
it is sufficient to see that the homomorphism  $\pi_2(S,s)\to \pi_1(X_s,x)$ is trivial,
which follows as by   hyperbolic geometry
  $\pi_1(X_s,x)$ does not contain any normal non-trivial abelian
 subgroup. To see this classical fact, one may e.g. apply
   \cite[Thm.~2.3.6]{Kat92} and the triviality of the center of any finite index subgroup of $\pi_1(X_s,x)$.

\end{proof}

\section{An integral version of Deligne's Fixed Part Theorem}

\subsection{Integral Fixed Part Theorem} \label{ss:IFP}

In this section we recall the integral version of Deligne's Fixed Part Theorem in order to
motivate our  non-abelian version, Theorem~\ref{thm:main}. The reason why we
discuss the integral version of this classical result here is that while in the abelian case we can always make a
rational class integral by multiplication with a positive integer, this is not possible in the
non-abelian case. So one has to understand the integral results in the abelian world in
order to see what one can hope for in the non-abelian world.

Let $X $ and $S$ be quasi-projective complex manifolds and let $f\colon X\to S$ be a projective
holomorphic submersive map with connected fibers. Fix a point  $\tilde s$ in the universal cover $\tilde
S$  of $S$ and let $s\in S$ be its image. We say that
{\it a property  essentially holds }
if  the property holds after replacing $S$ by  a  finite  cover  $S' \to S$  with
$\tilde S\to S'\to  S$, $s$ by the image of $\tilde s$ in $S'$ and $X$ by $X\times_S S'$.
We denote by  $X_s$ the  fiber over $s$.

\smallskip

Let $Z$ be a  quasi-projective complex manifold.  A class $\xi\in H^{2i}(Z,\Q)$ is said to be a {\it Hodge class} if
\[
  \xi \in \im H^{2i}(Z, \Z)\cap \im F^iH^{2i}(\overline Z, \C) \subset H^{2i}(Z,\C)
\]
where $F$ is the Hodge filtration and $\overline Z$ is a good compactification of $Z$, i.e.\ a
complex manifold which is projective
and contains $Z$ as a Zariski open submanifold \cite[Thm.~~3.2.5]{Del74}.

\smallskip

As $f$ is a topological fiber bundle, $\pi_1(S,s)$ acts on the cohomology $H^i(X_s,\Q)$. We denote by $\widehat S$  the formal completion  of $S$ along $s$ and by $\widehat X$ the
formal  completion of $X$ along $X_s$.

\begin{thm}[Deligne's Fixed Part Theorem] \label{thm:del_fixpart}  Let $f: X\to S$
   be a projective, submersive morphism of quasi-projective complex manifolds  with $S$ as in I) in Section~\ref{sec:intro}.
Let $s$ be a complex point of $S$.
Let $\xi\in H^{2i}(X_s, \Q)$ be a Hodge class
on  $X_s$.
Then the following conditions are equivalent:
\begin{itemize}
\item[1)]   the orbit $\pi_1(S,s) \cdot \xi$ is finite;
\item[2)]  $\xi$ essentially extends to a  class in $H^{2i}(X,\C)$;
\item[3)]   $\xi$ essentially  extends to a  Hodge class in $H^{2i}(X,\Q)$;
\item[4)]  $\xi$  extends to  $H^{2i}(X_\Delta,\Omega^{\ge i}_{X_\Delta})$, where $\Delta$ is a
  contractible open neighborhood of $s$;
\item[5)] the Gau{\ss}-Manin flat deformation of $\xi$ in $H^{2i}_{dR}(\hat X/\hat S)$ lies in the subbundle  $$F^iH^{2i}(\hat X/\hat S):=H^{2i}(\hat X/\hat S, \Omega^{\ge i}_{\hat X/\hat S}) \subset  H^{2i}_{dR}(\hat X/\hat S).$$
\end{itemize}
\end{thm}

\begin{proof}
  The implications 3) $\Rightarrow$ 2) $\Rightarrow$ 1) and 3) $\Rightarrow$  4) $\Rightarrow $  5)
  are clear.

  \smallskip

  1) $\Rightarrow$ 3) holds rationally by Deligne's Fixed Part Theorem
  \cite[Thm.~~4.1.1]{Del71}, that is for any good compactification $\overline X$ of
  $X$ there exists an element in $ F^i H^{2 i}(\overline X,\C)\cap H^{2i}(\overline X,\Q)$
  extending $\xi$; let $\zeta $ be its image in  $ H^{2i}(X,\Q)$.

  We have to check that after replacing $S$ by a finite
  covering  $S' \to S$   with  $\tilde S \to S'\to S,$  the class $\zeta$ becomes integral. The obstruction
  for integrality is the image of $\zeta $ in $H^{2i}(X,\Q/\Z)$. There exists an $N>0$
  such that this image is induced by an element of
\begin{equation}\label{eq:abdelpf}
\ker  [ H^{2i}(X,\frac 1 N \Z / \Z )  \to   H^{2i}(X_s,\frac 1 N \Z / \Z )]
\end{equation}
We claim that there exists a finite \'etale covering $S'\to S$ such that the pullback along
$X\times_S S' \to X$ kills the group~\eqref{eq:abdelpf}. By the Leray spectral sequence it
is sufficient to show that we can kill the groups $H^p(S,R^q f_* (\frac 1 N \Z / \Z) )$ for $p+q=2i$
 and $p=1$, then for $p=2$ etc.\ by a
finite \'etale covering of $S$. This is Proposition~\ref{prop:artKpiki}.

\smallskip

We now address 5)  $\Rightarrow$ 1).
Recall that the  {\it Hodge} or {\it Noether-Lefschetz locus} $\mathrm{NL}$ is the closed
analytic subset of the \'etale space associated to $R^{2i}f_* \Q$ consisting of Hodge
cycles, see~\cite[Intro.]{CDK95}. Restricted to any irreducible component of $\mathrm{NL}$ the morphism $\mathrm{NL}\to S$ is finite and unramified.
The assumption 5) implies that there exists only one irreducible component $W $ of  $ \mathrm{NL}$ containing the given Hodge
cycle $\xi$ and that $W\to S$ is \'etale at $\xi$.  Then $W\to S$
 is a finite, surjective, unramified holomorphic map of reduced, irreducible analytic spaces of the same
 dimension with a complex manifold as codomain, so it is a finite \'etale covering of complex manifolds.
 We conclude that the orbit $\pi_1(S,s) \cdot \xi$ consists of the finitely many elements of the fiber of
 $W\to S$ over $s$.  See also \cite[Intro.]{BEK14} and references therein.
\end{proof}

\section{Extensions of representations of  discrete groups} \label{sec:extreps}

This section contains the main group theoretic results of our note, which are related to
Clifford theory~\cite{Clif}.
Let $G$ and $H$ be finitely generated groups such that  $H\subset G$ is a normal subgroup.
We write $\Gamma$ for the group $G/H$. Let $K$ be a field of characteristic zero and $V$  be a
finite dimensional $K$-vector space.

\subsection{Main group theoretic result}

We say that a property holds {\it essentially} for a representation $\rho \colon G\to \GL(V)
 $ if it holds after restriction to a subgroup $G'$ of finite index in $G$ with $H\subset
 G'$. We say that a representation $\rho_H\colon H \to \GL(V)$ {\it essentially extends} to $G$ if it extends to a
 representation of such a $G'$.

For a representation $\rho \colon G\to \GL(V)$ we
denote by $\mathrm{Mo}(\rho)$ the Zariski closure of the image of $\rho$ as an
algebraic group over $K$. Its identity component $\mathrm{Mo}(\rho)^0$ is called the
{\it algebraic monodromy group} of $\rho$. We call $\rho$ {\it algebraically isomonodromic
relative to $H$} if the  embedding
\[
\mathrm{Mo}(\rho|_H )^0 \to \mathrm{Mo}(\rho)^0
\]
defined by $H\to G$
is an isomorphism.

We denote by $[\rho_H]\in \mathrm{Rep}(H)$ the isomorphism class of a representation
$\rho_H \colon H\to \GL (V)$. There is a natural action of $\Gamma$ on the set of
isomorphism classes of representations $\mathrm{Rep}(H)$ by conjugation by lifts to $G$ of elements in  $\Gamma$.

\begin{thm}\label{thm:maingroup}
Assume that the map of profinite completions $\widehat H\to \widehat G$ is injective. Let
$\rho_H\colon H \to \GL(V)$ be a representation with semisimple algebraic monodromy
group.
The following conditions are equivalent.
\begin{itemize}
\item[a)] the orbit $\Gamma \cdot [\rho_H]$ is finite;
\item[b)]  $\rho_H$ essentially extends to a representation
    $\rho\colon G \to \GL(V)$;
\item[c)] $\rho_H$ essentially extends to an algebraically isomonodromic representation
  $\rho\colon G \to \GL(V)$.

\end{itemize}
Moreover,
\begin{itemize}
\item[i)]  if  a representation  $\rho$ as in c) exists it is essentially unique;
\item[ii)]  if $K$ is a number field and $\rho_H$
is integral then a representation $\rho$ as in c) is integral;
\item[ iii)] if $\rho_H$ is absolutely simple, then an extension $\rho: G\to \GL(V)$ of  $\rho_H$  is algebraically  isomonodromic relative to $H$  if and only if  $\rho$  has finite determinant.
\end{itemize}
\end{thm}

The proof of Theorem~\ref{thm:maingroup} is given in Subsection~\ref{subsec:algisoreps}.

\subsection{Admissible representations}

{In this subsection we assume that $K$ is algebraically closed of characteristic zero.} A representation
$\rho_H\colon H\to \GL(V)$ is called {\it admissible} if it is semisimple and all simple
constituents of $\rho_H$ have finite determinant. A representation
$\rho\colon G\to \GL(V)$ is called {\it admissible relative to $H$} (or an {\it admissible
  extension of $\rho|_H$}) if $\rho$ is admissible and all its simple
constituents
stay simple when restricted to $H$.

\smallskip

The two  following lemmata  ought to be well known, but we could not find a precise reference.

\begin{lem}\label{lem:semi-simplealg}
The algebraic monodromy group $\mathrm{Mo}(\rho)^0$ of $\rho$ is semisimple if and only if $\rho|_{G'}$ is
admissible for all subgroups of finite index $G'\subset G$.
\end{lem}

\begin{proof}
  Assume the algebraic monodromy group $\mathrm{Mo}(\rho)^0$ is semisimple.  Then for all finite index subgroups $G'\subset G$, $\rho(G')\subset \rho(G)$ has finite index, thus  $\mathrm{Mo}(\rho|_{G'}) \subset  \mathrm{Mo}(\rho)$
  has finite index as well.  This implies the equality $\mathrm{Mo}(\rho|_{G'})^0 =  \mathrm{Mo}(\rho)^0$.

   In particular,  $\mathrm{Mo}(\rho|_{G'})^0$
   is reductive, thus by Weyl's theorem, $\rho|_{G'}$
 as a   representation of the
  algebraic group ${\rm Mo}(\rho|_{G'})^0$ is semisimple, thus $\rho|_{G'}$ itself is semisimple. As  ${\rm Mo}(\rho)^0= [ {\rm Mo}(\rho)^0,{\rm Mo}(\rho)^0 ]$,  all rank one algebraic
representations of ${\rm Mo}(\rho|_{G'})$ are finite.  Thus, $\rho|_{G'}$ is admissible.

  \smallskip

  Assume conversely that $\rho|_{G'}$ is admissible for all subgroups of finite index
$G'\subset G$. Upon replacing $G$ by the preimage of  $ {\rm Mo}(\rho)^0(K)$, which is a finite index subgroup  in $G$ (recall $K$ is algebraically closed),
we may assume that
$\mathrm{Mo}(\rho)$ is connected.  As the unipotent radical $U$ of  $\mathrm{Mo}(\rho)$ is normal, by Clifford's theorem
$\rho|_U$ is still semisimple. So by  Engel's theorem, we conclude that $U$ is trivial.
The radical $R$  of  $\mathrm{Mo}(\rho)$ is then a
central torus, so by Schur's lemma it acts by a character $\lambda\colon R\to \mathbb G_m$ on each simple
$G$-subrepresentation  $V_i \subset V$. As the action on $\det(V_i)$ is finite,  a power of
$\lambda$ is trivial, so $\lambda$ is trivial.  As  the  action of $R$ on $V$ is faithful,
  $R$ is trivial and ${\rm Mo}(\rho)^0$ is semisimple.
\end{proof}

 We draw the following consequence of Lemma~\ref{lem:semi-simplealg} which ought to be well known, but we could not find a reference.
\begin{lem}
Finite direct sums of representations with semisimple algebraic monodromy groups have semisimple
algebraic monodromy groups. Direct summands of representations with semisimple algebraic
monodromy group have a semisimple algebraic monodromy group.
\end{lem}

\begin{proof}
By Lemma~\ref{lem:semi-simplealg} it is sufficient to prove this for ``semisimple
algebraic monodromy group'' replaced by ``admissible  in restriction to all finite index subgroups'', for which it is trivial.
\end{proof}

\begin{lem}\label{lem:admend}
Let $ \rho \colon G\to \GL(V)$ be an   admissible representation relative to
$H$, then  the restriction homomorphism
\[
\mathrm{End}_\rho (V) \xrightarrow\sim \mathrm{End}_{\rho|_H} (V).
\]
is an isomorphism.
\end{lem}

\begin{proof}
Let $ V = \bigoplus_i V_i \otimes W_i$ be the canonical decomposition with $V_i$ non-isomorphic
simple representations of $G$ and $W_i$ $K$-vector spaces as trivial representations. By Schur's lemma
\[
\mathrm{End}_\rho (V) = \prod_i  \mathrm{End}(W_i).
\]
As $V_i$ as a representation of $H$ is simple as well,
the same formula holds for $\mathrm{End}_{\rho|_H} (V)$.
\end{proof}

\begin{rmk}
In Lemma~\ref{lem:admend} we do not need the condition on the determinants of the constituents being torsion. We won't need this fact.
\end{rmk}

\begin{prop}\label{prop:uniqextadmis}
Let $\rho,\tilde \rho\colon G\to \GL(V)$ be two admissible representations relative to
$H$ with $\rho|_H = \tilde\rho|_H$. Then $\rho$ is essentially isomorphic to $\tilde \rho$.
\end{prop}

\begin{proof}
A vector subspace $V'\subset V$ is a simple $\rho$-subrepresentation if and only if it is the image of $V$  by a
minimal right ideal of $\mathrm{End}_\rho$ and the same for $\rho$ replaced with $\tilde
\rho$. So by Lemma~\ref{lem:admend} the simple $\rho$-subrepresentations are exactly the
simple $\tilde \rho$-subrepresentations. We can therefore assume that $\rho$ and $\tilde
\rho$ are simple.

Set $\tau(g)=\tilde \rho(g) \rho(g)^{-1}$ for $g\in G$. As
\ga{}{ \tau(g) \rho(h) \tau(g)^{-1} =
 \tilde{ \rho}(g) \rho(g^{-1}hg)\tilde{ \rho}(g^{-1})
 = \tilde{ \rho}(g)
 \tilde{ \rho}(g^{-1}hg) \tilde{\rho}(g^{-1})=
\rho(h) \notag }
 for all $h\in H$ and $g\in G$,  and as   $\rho|_H$ is simple, Schur's lemma implies that $\tau(g)\in K^\times $
for all $g\in G$. In fact $\tau(g)\in \mu_N(K)$ for some $N>0$ as the determinants of
$\rho$ and $\tilde \rho$ are finite.   This  implies that $\tau \colon G\to \mu_N(K)\subset \GL(V)$ is a homomorphism. Thus, $\rho=\tilde \rho$ in restriction to
 the kernel $G' \subset G $ of $\tau $, which is a finite index subgroup. \end{proof}

The next proposition is the reason why we have to pass to profinite completions.
\begin{prop}\label{prop:existextadmis}
  Assume that the map of profinite completions $\widehat H\to \widehat G$ is injective.
 For an admissible representation  $\rho_H\colon H\to
 \mathrm{GL}(V)$ the following are equivalent:
 \begin{itemize}
 \item[1)] the orbit
$\Gamma \cdot [\rho_H]$ is finite;
\item[  2)] $\rho_H$ essentially extends to a representation of $G$;
\item[ 3)] $\rho_H$  essentially extends to an admissible representation $\rho$ of $G$ (relative to $H$).

\end{itemize}
\end{prop}

\begin{proof}
3) $\Rightarrow $ 2) $\Rightarrow$ 1) is obvious. We have to show 1) $\Rightarrow $
3). After replacing $G$ by a subgroup $G'$ of finite index with $H\subset G'$  the orbit
$\Gamma \cdot [\rho_H]$ is trivial.  As  $\Gamma$ permutes the finitely many
isomorphism classes of simple
constituents of $\rho_H$,  after replacing  $G$ by a subgroup $G'$ of finite index with $H\subset G'$,  we may assume that this permutation
action is trivial.
Then we may and do  assume that $\rho_H$ is simple.  Choose a basis of $V$ and identify
 $\GL_r(K)$ with $\GL(V)$.  As $H$ is finitely generated there exists a finitely generated
 $\mathbb Z$-subalgebra $R\subset K$ with $\im(\rho_H)\subset \GL_r(R)$. Choose an injective ring
 homomorphism $R\to \sO_E $ where $\sO_E\subset E$ is the ring of integer of an $\ell$-adic field $E$
  for some  prime number $\ell$.
 The induced representation $\rho_H\colon H\to \GL_r(E)$ factors
 through a continuous representation of $\widehat H$, which is simple as well. By Proposition~\ref{prop:appext} we can  essentially extend it to a continuous
 representation $\widehat \rho\colon \widehat G\to \GL_r(E)$ with finite determinant.
 As $G$ is finitely generated
 the image of $\widehat \rho|_G$  has values  in  $\GL_r(A)$, where $A$ is a finitely generated
 $\mathrm{Frac}(R)$-subalgebra of $E$. Choose a $K$-point of $\Spec(A)$ over
 $\mathrm{Frac}(R)$ (recall that $K$ is assumed to be  algebraically closed).  It induces an extension $\rho\colon G\to \GL_r(K)$ of $\rho_H$ with finite determinant.
\end{proof}

\subsection{Algebraically isomonodromic representations}\label{subsec:algisoreps}

In this subsection we collect some facts about algebraically isomonodromic representations
relative to $H$.

\begin{lem}\label{lem:algisocon}
  If $\rho $ is algebraically isomonodromic relative to $H$,  then
  $
\mathrm{Mo}(\rho|_H ) \to \mathrm{Mo}(\rho)
$
is essentially an isomorphism.
In particular,
$
\mathrm{End}_\rho(V) \xrightarrow\sim \mathrm{End}_{\rho|_H}(V)
$
is essentially an isomorphism.
\end{lem}

\begin{proof}
  Under the isomonodromic  assumption,  ${\rm Mo}(\rho|_H)\subset {\rm Mo}(\rho)$ is a
  (normal)  finite index algebraic subgroup, which induces a finite index (normal)  embedding
  $\pi_0( {\rm Mo}(\rho|_H) ) \subset  \pi_0({\rm Mo}(\rho))$  of finite groups. After
  replacing $G$ by the preimage of $\pi_0( {\rm Mo}(\rho|_H) )$ the map   $
\mathrm{Mo}(\rho|_H ) \to \mathrm{Mo}(\rho)
$ is an isomorphism.
    \end{proof}

  \begin{lem}\label{lem:algisoadmcom}
    Assume $K$ is algebraically closed.
If $\rho\colon G\to \GL(V)$ is algebraically isomonodromic relative to $H$ and
$\mathrm{Mo}(\rho)^0$ semisimple, then $\rho$ is essentially admissible relative to $H$.
\end{lem}

\begin{proof}
By Lemma~\ref{lem:algisocon} we may assume that $\mathrm{Mo}(\rho|_H ) \to
\mathrm{Mo}(\rho)$ is an isomorphism. In that case the lemma is clear.
\end{proof}

\begin{prop}\label{prop:admalgisom}
  Assume $K$ is algebraically closed.
Let $\rho\colon G\to \GL(V)$ be a represenation such that $\mathrm{Mo}(\rho|_H)$ is
connected and semisimple. Then $\rho$ is essentially admissible relative to $H$ if and only if  $\rho$ is
algebraically isomonodromic relative to $H$.
\end{prop}

\begin{proof}
The implication ``$\Leftarrow$'' follows from Lemma~\ref{lem:algisoadmcom}.

\smallskip

For ``$\Rightarrow$'' we first observe that by replacing $G$ by the preimage of
$\mathrm{Mo}(\rho)^0$, which is a subgroup of finite index containing $H$, we  may assume
that $\mathrm{Mo}(\rho)$ is connected.
The argument in the proof of
Lemma~\ref{lem:semi-simplealg} shows that $\mathrm{Mo}(\rho)$ is semisimple. If the embedding
$\mathrm{Mo}(\rho|_H)\to \mathrm{Mo}(\rho)$ is not an isomorphism then there exists a
non-trivial
normal connected algebraic subgroup $W\subset
\mathrm{Mo}(\rho)$ centralizing  $\mathrm{Mo}(\rho|_H)$  such that
\[
W\times \mathrm{Mo}(\rho|_H) \to \mathrm{Mo}(\rho)
\]
is surjective with finite kernel, see e.g.~\cite[14.2]{Hum75}.
The $W$ action stabilizes  every simple $\rho$-subrepresentation
$V'\subset V$ and commutes with the action of $H$ on $V'$, which is simple. So by Schur's
lemma $W$ acts  by a character on  $ V'$, but all characters are trivial as $W$ is
semisimple. So $W$ has to be trivial as it acts faithfully on $V$.
\end{proof}

\begin{prop}\label{prop:finindexadmext}
  Assume that $K$  is algebraically closed. Let $\rho_H\colon H\to \GL(V)$ be a representation with
  semisimple algebraic monodromy group. Let $G'\subset G$ be a subgroup of finite index,
  $H'=H\cap G'$. If $\mathrm{Mo}(\rho|_{H'})$ is connected and $\rho|_{H'}$ has an admissible
  extension $\rho'\colon G'\to \GL(V)$ relative to $H'$, then $\rho$ essentially has an
  algebraically isomonodromic extension $\rho\colon G\to \GL(V)$ relative to $H$.
\end{prop}

\begin{proof}
After replacing $G'$ by a subgroup of finite index we may  assume that $G' \subset G$ is normal,
thus $H'\subset H$ is normal as well. Replace $G$ by the preimage of $H/H'\subset G/G'$ via $G\to G/G'$. Then by definition
 $[G:G']=[H:H']$. We set
 $$\tilde \rho=\mathrm{Ind}^G_{G'}(\rho').$$
 We show in the sequel that there is a direct summand of $\tilde \rho$ which extends  $\rho_H$.
 First observe that $\tilde \rho$  is algebraically isomonodromic
relative to $H$. Indeed, it suffices to check that $\tilde \rho|_{G'}$ is algebraically
isomonodromic  relative to $H'$.  Noting that
$\tilde\rho|_{H'}$ is a direct sum of copies of $\rho_H|_{H'}$, so
$\mathrm{Mo}(\tilde\rho|_{H'})$ is connected and semisimple, by Proposition~\ref{prop:admalgisom} we have  to show that
$\tilde
\rho|_{G'}$ is admissible relative to $H'$. This holds as $\tilde
\rho|_{G'}$
 it is a direct sum of
$H$-conjugates of $\rho'$ and $\rho'$ is admissible relative to $H'$.
On the other hand,
\[
\tilde \rho|_H \cong \mathrm{Ind}^H_{H'}(\rho_H|_{H'})
\]
contains $\rho_H$ as a direct summand. So by Lemma~\ref{lem:algisocon} after replacing
$G$ by a subgroup of finite index containing $H$ we find a direct summand of $\tilde \rho$
which extends $\rho_H$.
\end{proof}

\begin{proof}[Proof of Theorem~\ref{thm:maingroup}]
Essential uniqueness in   i) follows from Lemma~\ref{lem:algisoadmcom} and
Proposition~\ref{prop:uniqextadmis}.

\smallskip

We clearly have c) $\Rightarrow $ b) $\Rightarrow$ a).
We now reduce the proof of the implication a) $\Rightarrow $ c)  to the case where $ K$  is algebraically closed.  Let $\overline K$ be an algebraic closure of $K$.
Consider an algebraically isomonodromic representation $\rho \colon G\to \GL(V\otimes_K
\overline K)$ relative to $H$ with $\rho|_H$ stabilizing $V$ (thus defined over $K$). Then, as $G$ is finitely
generated, there exists a finite Galois subextension $\tilde K\subset \overline K$ of $K$
 such that $\rho$
stabilizes $V\otimes_K \tilde K$.   The Galois group  ${\rm Gal}(\tilde K/K)$ acts on $V\otimes_K \tilde K$ via its action on $\tilde K$, thus on $\GL(V\otimes_K \tilde K)$. For $\gamma\in {\rm Gal}(\tilde K/K)$ we define $\gamma\cdot \rho$ by the formula $\gamma\cdot \rho(g)=\gamma \cdot (\rho(g))$. Then ${\rm Mo}( \gamma\cdot \rho)=\gamma\cdot {\rm Mo}(\rho) $ thus both $\gamma\cdot \rho$  and $\rho$ are algebraically isomonodromic extensions relative to $H$ of $\rho|_H =\gamma\cdot \rho|_H$.  Then i) implies that there is a finite index subgroup $G^\gamma \subset G$ containing $H$ such that
$\rho|_{G^\gamma}= \gamma\cdot \rho|_{G^\gamma}$. We define $G'=\cap_{\gamma\in {\rm Gal}(\tilde K/K)} G^\gamma$ which contains $H$ and is a subgroup of finite index in $G$. Then $\rho|_{G'}$ is ${\rm Gal}(\overline K/K)$  invariant thus descends to $K$.

\smallskip

Now we prove  a) $\Rightarrow $ c) for $K =\overline K$. The condition $\widehat H \to
\widehat G$ injective allows us  to find a subgroup of finite
index $G'\subset G$, $H'=H\cap G'$, such that $\mathrm{Mo}(\rho_H|_{H'})$ is connected.
After replacing $G'$ by a subgroup of finite index containing $H'$,
Proposition~\ref{prop:existextadmis}   implies  that there  exists an admissible extension
$\rho'\colon G'\to \GL(V)$ of $\rho_H|_{H'}$ relative to $H'$.  Proposition~\ref{prop:finindexadmext}
shows then the essential existence of an algebraically isomonodromic extension $\rho\colon
G\to\GL(V)$ of $\rho_H$.

\smallskip

We now prove ii).
After replacing $K$ by a finite extension, we may assume that all simple constituents of $\rho_H$ are
absolutely simple. By Lemma~\ref{lem:algisocon} we can replace $G$ by a subgroup of finite
index containing $H$ such that $\mathrm{End}_\rho(V) \xrightarrow\sim \mathrm{End}_{\rho_H}(V)$ is an
isomorphism. Then we have to show ii) only for the simple subrepresentations $V'\subset
V$. So we  assume now that $\rho$ is absolutely simple and that $\mathcal V\subset V$ is a
$\rho_H$-stable lattice.

For $g\in G$ the lattice $\rho(g) \mathcal V$ is stabilized by the action of $H$ via
$\rho_H$   as $\rho(h)\rho(g)\mathcal V=\rho(g) \rho(g^{-1} h g) \mathcal V= \rho(g) \mathcal V$.
By the  the Jordan-Zassenhaus theorem \cite[Sec.~79]{CR62}  there are only finitely many such lattices up to
isomorphism,  which by Schur's lemma means up to homothety. So after replacing $G$   by a subgroup
of finite index containing $H$  we can assume that for all  $g\in G$ there is $\lambda_g\in
K^\times$ with
 $  \rho(g) \mathcal V = \lambda_g \mathcal V$.
As $\rho$ has finite determinant,   $\lambda_g\in
  \sO_K^\times$ and $\mathcal V$ is stabilized by the $\rho$-action of $G$.

\smallskip

We now prove iii). Without loss of generality $K=\overline K$. By assumption $\rho_H$ is absolutely simple and has finite determinant, thus
is admissible and the extension $\rho\colon G\to \mathrm{GL}(V)$ of $\rho_H$ with finite determinant
is admissible relative to $H$. Such an extension is essentially unique by Proposition~\ref{prop:uniqextadmis}.
But an algebraically isomonodromic extension of $\rho_H$, which essentially exists by the above argument, has finite determinant, so has to
essentially agree with $\rho$.
\end{proof}

\begin{prop}
  Assume that $\widehat H\to \widehat G$ is injective. Then
finite direct sums of algebraically isomonodromic representations with semisimple
algebraic monodromy are algebraically isomonodromic. Direct summands of algebraically isomonodromic representations with semisimple
algebraic monodromy are algebraically isomonodromic.
\end{prop}

\begin{proof}[Sketch of proof]
  The second part is easy. For the first part
we can assume without loss of generality that $K$ is algebraically closed, and we can argue
as in the proof of Theorem~\ref{thm:maingroup} reducing to the observation that direct
sums of admissible representations relative to $H$ are admissible relative to $H$.
\end{proof}

\section{Reminder on variations of Hodge structure}

In this section we summarize what we need about variations of Hodge structure. We could
not find some of the formulations in the literature, so we provide a few details.

\subsection{Polarized variations of $\C$-Hodge structure}\label{subsec:basichodgevar}

For our purpose it is useful to define polarized $\C$-Hodge structure in a form
without fixing a weight.
Let $V$ be a finite dimensional $\mathbb C$-vector space, $\mathsf F$ a finite decreasing
filtration of $V$ and $Q\colon V\times \overline V\to \C$ a perfect hermitian pairing.
We call the triple $(V,\mathsf F, Q)$ a {\it polarized $\C$-Hodge structure} ($
\C\text{-}{\sf PHS}$) if the {\it Hodge-Riemann relations} hold for   all $a\in \Z$:
\begin{itemize}
    \item  $V=\mathsf F^a\oplus (\mathsf F^a)^\perp$ and
 \item  $(-1)^a Q|_{\mathsf F^a\cap (\mathsf F^{a+1})^\perp} $  is positive definite.
\end{itemize}

If $N\colon V\to V$ is a nilpotent  endomorphism with
$Q(N v,w)=Q(v,N w)$ for all $v,w\in V$ then we call $(V,\mathsf F, Q,N)$  a {\it
  polarized mixed  $\C$-Hodge structure} if $N(\mathsf F^a)\subset \mathsf F^{a-1}$
and if the filtrations  induced by $\mathsf F$ on
 the monodromy graded pieces $\gr^{\mathsf M}_i V$ are compatible with the Lefschetz
 decomposition of $\oplus_i \gr^{\mathsf M}_{i} V$ with respect to $\overline N\colon
 \gr^{\mathsf M}_{i} V\to  \gr^{\mathsf M}_{i-2} V$
 and furthermore  for $i>0$ the primitive part
\[
\ker ( \gr^{\mathsf M}_i V \xrightarrow{\overline N^{i+1}}\gr^{\mathsf M}_{-i-2} V)
\]
endowed with
the pairing
\[
  \gr^{\mathsf M}_i V\times \gr^{\mathsf M}_i V\to \C, \quad
  (v,w)\mapsto Q(v,\overline N^i w)
\]
and the induced filtration is a $\C\text{-}{\sf PHS}$ .

Let us recall how this notion of a $\C\text{-}{\sf PHS}$  is related to the more common
version of a   $ \R\text{-}{\sf PHS}$.
Let $V$ be a finite dimensional $K$-vector space with $K\in \{ \mathbb Q, \mathbb
R\}$, $\mathsf F$ a finite decreasing filtration on $V_\C=V\otimes_K \mathbb C$ and $Q_\circ\colon
V\times V\to K$ a $(-1)^w$-symmetric perfect pairing with $w\in \Z$. We
call  the triple $(V,\mathsf F, Q)$ a {\it  polarized $K$-Hodge structure} ($
K\text{-}{\sf PHS}$) of {\it  weight $w$} if
\[
  (V_\C , \mathsf F, (v,w)\stackrel{Q}{\mapsto} \sqrt{-1}^{-w }
  Q_\circ(v,\overline w))
\]
is a $\C\text{-}{\sf PHS}$ and $  (\mathsf F^a)^\perp = \overline{\mathsf F}^{w+1-a} $
for all $a\in\Z$,
where the orthogonal complement is with respect to the hermitian pairing $Q$.

 A {\it polarized variation of $\C$-Hodge structure}
($\C\text{-}{\sf PVHS}$) on a complex manifold $X$  is given by a triple $(\mathbb L,
\mathsf F, Q)$, where $\mathbb
L$ is a complex local system on $X$, $\mathsf F\subset \mathbb L\otimes_\C \mathcal \sO_X$ is a finite
filtration by holomorphic subbundles and $Q\colon \mathbb L\times \overline{\mathbb L}\to \C$ is a
hermitian perfect pairing, where we assume that the Hodge-Riemann relations
at each point of $X$  are  satisfied and
Griffiths transversality $\nabla (\mathsf F^a) \subset \Omega^1_X(\mathsf F^{a-1})$ holds.

\begin{thm} \label{thm:ss} Let $X$ be a quasi-projective connected complex  manifold. Let $(\mathbb
  L,\mathsf F,Q)$ and $(\mathbb L', \mathsf F',Q')$ be $\C\text{-}{\sf PVHS}$. The following properties hold:
  \begin{itemize}
    \item[i)] if $\phi\colon \mathbb L\xrightarrow\sim \mathbb L'$ is an isomorphism of local systems such
      that for one point $x\in X$ the Hodge filtrations are preserved, i.e.\
      $\phi_x(\mathsf  F_x^a)= (\mathsf F')_x^a$ for all $a\in \Z$, and such that $\phi^* (Q')=Q$,  then
      $\phi$ is an isomorphism of $\C\text{-}{\sf PVHS}$;
    \item[ii)] if $\mathbb L$ is simple then  $Q$ is unique up to a factor in $\mathbb
      R_{> 0}$ and $\mathsf F$ is unique up to shift. Moreover, there are no gaps in
      $\mathsf F$, i.e.\ we cannot have $\gr^{\mathsf F}_i=0$ and  $\gr^{\mathsf
        F}_j,\gr^{\mathsf F}_k\ne 0$ for
      $j<i<k$;
      \item[iii)] there is a canonical structure of a $\C\text{-}{\sf PHS}$ on
        $H^0(X,\mathbb L)$ such that  $H^0(X,\mathbb L)\to
        \mathbb L_x$ is a polarized Hodge substructure for all $x\in X$;
   \item[iv)] there are $\C\text{-}{\sf PVHS}$ $(\mathbb L_i,\mathsf F_i,Q_i)$ for $1\le i\le s$ with
  $\mathbb L_i$  simple local systems such that the canonical map
  \[
    \bigoplus_i (\mathbb L_i ,\mathsf F_i,Q_i   )\otimes_\C   H^0(X, \mathbb L_i^\vee \otimes \mathbb
    L) \xrightarrow{\sim}  (\mathbb L,\mathsf F,Q)
  \]
  is an isomorphim in $\C\text{-}{\sf PVHS}$. Here $H^0$ is endowed with the
  $\C\text{-}{\sf PHS}$   from~{\rm iii)};
\item[v)] if $\mathbb L$ underlies a $\mathbb Z$-local system then
  the algebraic monodromy group  $\mathrm{Mo}(\mathbb L,x)^0$ is semisimple and the
  monodromy at infinity of $\mathbb L$ is quasi-unipotent.
 \end{itemize}
\end{thm}

The theorem is classical~\cite{Del87}.

\begin{prop}\label{prop:qfinhodgevar}
Let $f\colon Y\to X$ be either a finite \'etale covering or a Zariski open  embedding (both with
 dense image) of quasi-projective complex manifolds. Let $\mathbb  L$ be a $\C$-local
system on $X$. Then $\mathbb L$  underlies a $\C\text{-}{\sf PVHS}$ if and only if
$f^*\mathbb L$ underlies a $\C\text{-}{\sf PVHS}$.
\end{prop}

\begin{proof}
The case of a finite \'etale covering is a consequence of Theorem~\ref{thm:ss}(iv) which
implies that for $\mathbb L_1 \subset \mathbb L_2$ an  inclusion of complex local systems such
that $\mathbb L_2$ underlies a $\C\text{-}{\sf PVHS}$  then $\mathbb L_1$
underlies a $\C\text{-}{\sf PVHS}$.

The case of a
Zariski open embedding follows from the nilpotent orbit theorem~\cite[(2.1)]{CK89}.
\end{proof}

Let us recall for later reference one of the main results about {\it degeneration of pure
  $\C$-Hodge structure}. Let  $\mathbb L$ be a unipotent complex local system on  $\Delta^\times=\Delta\setminus \{0 \}$ with $\Delta= \{ z\in \C\, |\,  |
z| <1 \} $. Let $Q\colon \mathbb L \times\overline{\mathbb L}\to \C$ be a hermitian perfect
pairing. Let $\mathcal E$ be the holomorphic bundle on $\Delta$  which is the Deligne extension of the
flat bundle $\mathbb L \otimes \sO_{\Delta^\times}$ \cite[Prop.~II.5.2]{Del70}. Let
$N\colon \mathcal E_0\to \mathcal E_0$ be the residue of the flat logarithmic connection
$\nabla\colon \mathcal E\to \frac 1 z \Omega^{1}_X(\mathcal E)$ on $\Delta$.
  One can construct an induced hermitian pairing $Q_0\colon \mathcal E_0\times
\overline{\mathcal E}_0\to \C$ using a (tangential) base point.  Note that our $N$ differs
from  the one in~\cite{SS22} by a sign.

\begin{prop}\label{prop:deghodgestr}
Let $\mathsf F\subset \mathcal E|_{\Delta^\times}$ be a Griffiths transversal finite
filtration by holomorphic subbundles.  Then the following are equivalent:
\begin{itemize}
\item[1)] for $0<|z|\ll 1$,
 $(\mathcal E_z, \sF_z, Q_z$) is a pure polarized Hodge structure;
\item[2)] $\mathsf F$ extends to a filtration by holomorphic subbundles of $\mathcal E$
  and
  $(\mathcal E_0, \sF_0, N, Q_0)$ is a polarized mixed Hodge structure, i.e.\ it  is a 
  nilpotent orbit
  by~\cite[(2.3)]{CK89}.
\end{itemize}
\end{prop}

\begin{proof}
1) $\Rightarrow $ 2) is  Schmid's nilpotent orbit theorem, see~\cite[(2.1)]{CK89}. \\
2) $\Rightarrow $ 1) is shown in \cite[Thm.~2.8]{CK89} even in the several variable case.
\end{proof}

\begin{rmk}\label{rmkdeghodgunifo}
There is a version of Proposition~\ref{prop:deghodgestr} with continuous dependence on
parameters. We later need    2) $\Rightarrow$ 1) with parameters, i.e.\ if the family of holomorphic filtrations
$(\mathsf F(s))_{s\in V}$ of $\mathcal E$ depends continuously on parameters $s\in V$ with
$V\subset \mathbb R^m$ open, $0\in V$,  and for every $s\in V$ the filtration $\mathsf F_{0}(s)$ is a mixed
Hodge structure on $\mathcal E_0$ polarized by  $N$ and $Q_0$,   then
 there exists $\epsilon >0$ and a neighborhood $U\subset  V$ of $0$  such that  $ (\mathcal E(s), \mathsf
 F_z(s))$  is a pure Hodge structure polarized by $Q_z$ for $0<|z|<\epsilon $ and $s\in U$,
see the argument in \cite[Sec.~4]{CK89} which extends immediately to the parametrized situation. \end{rmk}

  Let now $X$ be a Riemann surface with compactification $\overline X$ and let
  $(\mathbb L,\mathsf F,Q)$ be a $\C\text{-}{\sf PVHS}$ on $X$. Assume that the local
  monodromies at infinity of $\mathbb L$ are unipotent. Then the anti-holomorphic Deligne
  extension $\overline{\mathcal E}$ of the complex conjugate local system
  ${\overline{\mathbb L}} $ is the complex conjugate bundle of $\mathcal E$.
  By \cite[App.B]{EV86}  ${\rm deg}( \overline{ \mathcal E)}={\rm deg}( \mathcal E)=0$.
  By the unipotence of local monodromy the formation of the Deligne extension is compatible
with the dual bundle.

By Schmid's nilpotent orbit theorem $\mathsf F^a$ extends to a holomorphic
subbundles of $\mathcal E$ and its orthogonal complement $\overline{\mathsf F}^a= (\mathsf
F^{-a})^\perp$ extends to an anti-holomorphic subbundle of $\overline{E}$, which we denote
by the same symbols.
 So
  $\overline{\mathsf F}^a$ is the complex conjugate dual of $\mathcal E/{\mathsf F}^{-a}$. We
  conclude

\begin{lem}\label{lem:degsum}
For all integers $a$ it holds
  \[
    \deg ( {\mathsf F}^a) + \deg (\overline{\mathsf F}^{-a} ) =0.
  \]
\end{lem}

\begin{rmk}\label{varhodgermkqu}
 Note
that  in  the classical literature, see~\cite{CK89}, the degeneration of Hodge structure, and
as a consequence Deligne's semisimplicity theorem, are studied for  $\R\text{-}{\sf PVHS}$  with quasi-unipotent monodromy at infinity.
Deligne remarked that the theory generalizes to arbitrary  $\C\text{-}{\sf PVHS}$~\cite[1.11]{Del87}.
Recently, this was documented in the one variable case in~\cite{SS22}.
\end{rmk}

\subsection{ Characterization of a polarized variation of $\C$-Hodge structure via
  the non-abelian Hodge correspondence} \label{ss:mochi}

Let $\overline X$ be a projective complex  manifold (with fixed ample line bundle). Let $D\subset \overline X$ be a
simple normal crossings divisor with complement $X$. We recall following Simpson~\cite[Lem.~4.1]{Sim92}
and T.~Mochizuki~\cite[10.1]{Moc06} a citerion for a simple complex local system $\mathbb L$
on $X$ to underly a $\C\text{-}{\sf PVHS}$. Let $E$ be the $\mathcal C^\infty$-bundle
underlying $\mathbb L$ with its flat connection $\nabla$.

T.~Mochizuki and Jost--Zuo construct a tame, purely imaginary pluri-harmonic
metric $h$ on $E$ which is unique up to a factor in $\mathbb R_{>0}$, see
\cite[Part~5]{Moc07} and~\cite{Moc09}. This allows us to
write $\nabla= \partial_E + \overline\partial_E + \theta +  \theta^\dagger$. Here
$\theta$ is a $(1,0)$-form with values in $\mathit{End}(E)$ with adjoint  $
\theta^\dagger $ with respect to $h$ and $\partial_E + \overline\partial_E $ is a connection compatible
with $h$. The associated Higgs bundle $(E,\overline\partial_E , \theta)$ has a canonical
parabolic structure at infinity such that $\theta $ has only simple poles and purely
imaginary eigenvalues of its residues.

For $\lambda\in S^1= \{ z\in \mathbb C\, |\, |z|=1 \} $ one considers the flat bundle $(E,\nabla^\lambda)$ with associated
local system $\mathbb L^\lambda$, where  $\nabla^\lambda = \partial_E +
\overline\partial_E + \lambda \theta +  \overline \lambda   \theta^\dagger$,
see~\cite[2.2.1]{Moc09}.

\begin{prop}\label{prop:varhodgecrit}
 Fix $\lambda \in S^1$ which is not a root of unity. Then
$\mathbb L\cong \mathbb L^\lambda$ if and only if $\mathbb L$ underlies a $\C\text{-}{\sf PVHS}$.
\end{prop}

\begin{proof}
The Higgs bundle of $\mathbb L^\lambda$ is $(E,\overline \partial_E ,\lambda \theta)$ with
the same parabolic structure as the Higgs bundle of $\mathbb L$. So $\mathbb L\cong
\mathbb L^\lambda$ if and only if there is an isomorphism of Higgs bundles
\[
\phi\colon (E,\overline \partial_E ,\theta) \xrightarrow\sim (E,\overline \partial_E ,\lambda \theta)
\]
compatible with
the parabolic structure. Decomposing $E$ according to generalized eigenspaces of $\phi$
induces an $h$-orthogonal bundle decomposition $E= E_1 \oplus \cdots \oplus E_r$
with $\theta(E_i)\subset\Omega^1_X( E_{i-1})$, which itself corresponds to a $\C\text{-}{\sf PVHS}$ as
in~\cite[Lem.~4.1]{Sim92}.
\end{proof}

\subsection{The Mumford-Tate group of a polarizable variation of $\Q$-Hodge structure} \label{ss:MT}
To a polarizable $\Q$-Hodge structure $(V,\mathsf F)$ of weight $w$ one   associates the   {\it Mumford-Tate
  group} ${\rm MT}(V,\mathsf F)$. It can be characterized as the smallest linear algebraic subgroup
of $\mathrm{GL}(V)$ over $\Q$ the complex points of which comprise the maps
$\phi_\lambda\colon V_\C\to V_\C$   for $\lambda \in \mathbb C$
 defined  by $\lambda^i \overline \lambda^j $ on $ \mathsf F^i\cap \overline{\mathsf
   F}^j$, where $i+j=w$. In particular, ${\rm MT}(V,\mathsf F)$ is connected.

Let now $X$ be a quasi-projective complex manifold  and let $(\mathbb L ,\mathsf F)$ be a
polarizable variation of $\Q$-Hodge structure on $X$. We also assume that $\mathbb L$
underlies a $\mathbb  Z$-local system, which however should not be necessary for most arguments in view of Remark~\ref{varhodgermkqu}.

\begin{prop} \label{prop:norm}
The Mumford-Tate group ${\rm MT}(\mathbb L_x, F_x)$ normalizes the Zariski closure    ${\rm Mo}(\mathbb L, x)$  of the
monodromy group
in $\GL(\mathbb L_x)$.
\end{prop}

A slightly weaker result has been observed in~\cite[Sec.~5]{And92}. We argue in terms of
Tannaka groups.

\begin{rmk}\label{rmk:tannakagr}
 The Mumford-Tate group ${\rm MT}(V,\mathsf F)$ and the group ${\rm
  Mo}(\mathbb L, x)$ are Tannaka groups in view of the following observation.
Let $V$ be a $\mathbb Q$-vector space and $G\subset \mathrm{GL}(V_\C)$ a abstract subgroup.
There is an associated Tannaka category $\mathbf  T$. Its  objects  are all subquotients of $V^{\otimes
  n}\otimes (V^\vee)^{\otimes m}$ preserved over $\C$ by the $G$-action and its
morphisms are $\Q$-linear maps compatible with the $G$-action over $\C$. Then the Tannaka
group $\mathrm{Tann}(\mathbf T )$ with respect to the obvious fiber functor over $\Q$ is the same as the smallest linear
algebraic subgroup of $\mathrm{GL}(V)$ over $\Q$ such that its complexification contains $G$.
\end{rmk}

\begin{proof}[Proof of Proposition~\ref{prop:norm}]
  Recall that polarizable variations of mixed $\Q$-Hodge structure form a Tannaka category
  for which every point $x\in X$ induces a fiber functor. We think of any polarizable pure
  variation of $\Q$-Hodge structure as an object of this Tannaka category.
 We denote by  ${\rm Tann}((\mathbb L, {\sf F}), x) \subset \GL(\mathbb L_x)$  its Tannaka
 group with respect to the fiber functor at $x$.
  Note that ${\rm Mo}(\mathbb L, x)$ is the Tannaka group of $\mathbb L$ in the Tannaka  category of
 $\Q$-local systems, and that
   by Remark~\ref{rmk:tannakagr}
    the Tannaka
  group of a polarizable $\Q\text{-}$Hodge structure
   is its Mumford-Tate group.
 The inclusion of Tannaka groups
 \ga{}{   {\rm Mo}(\mathbb L, x)\subset  {\rm Tann}(({\mathbb{ L}}, {\sf F}), x )    \notag  }
  which corresponds to the functor of Tannaka categories
  defined by forgetting the Hodge filtration
  is normal by \cite[Cor.~1.2]{dAE22} and \cite[Thm.~A.1]{EHS07}.
The restriction to $x$ of a variation of Hodge structure   induces an inclusion of Tannaka groups
 \ga{5.1}{ {\rm MT}(\mathbb L_x, {\sf F}_x) \subset {\rm Tann}((\mathbb L, {\sf F}),x)  .}   This finishes the proof.

 \end{proof}

We define the {\em Mumford-Tate group  of the polarizable variation of  $ \Q$-Hodge structure}
$(\mathbb L,\mathsf F)$ at $x\in X$ as the product
\ga{}{ {\rm MT}((\mathbb L, {\sf F}),x) =  {\rm Mo}(\mathbb L,x)^0 \cdot  {\rm MT}(\mathbb L_x, {\sf F}_x) \subset \GL(\mathbb L_x),\notag}
which is a connected subgroup as by  Proposition~\ref{prop:norm} the second factor
normalizes the first factor.

For two polarizable variations of
$\Q$-Hodge structure $(\mathbb L_1 , \mathsf F_1)$ and $(\mathbb L_2 , \mathsf F_2)$   a
$\Q$-linear map $\psi\colon\mathbb L_{1,x}\to \mathbb L_{2,x}$ induces a morphism of variations of
Hodge structure if and only if $\psi$ is a morphism of Hodge structures at $x$ and if $\psi$ commutes with
the action of $\pi_1(X,x)$, see Theorem~\ref{thm:ss}(i).  So by Remark~\ref{rmk:tannakagr}
we deduce  that
 ${\rm MT}((\mathbb L, {\sf F}),x) $ is the
neutral component of
\[
  \mathrm{Tann}((\mathbb L,\mathsf F),x) =  {\rm Mo}(\mathbb L,x) \cdot   {\rm MT}(\mathbb L_x, {\sf F}_x) \subset \GL(\mathbb L_x).
\]
As an isomorphism of two fiber functors induces an isomorphism of Tannaka groups this shows that the isomorphism $\GL(\mathbb L_x) \cong \GL(\mathbb L_y)$, induced by a
path between $x,y\in X$, gives us an isomorphism  $ {\rm MT}((\mathbb L,
{\sf F}),x) \cong  {\rm MT}((\mathbb L, {\sf F}),y)$. In this sense {\it the Mumford-Tate
groups at different points of $X$
of a variation of Hodge structure on $X$ form a local system}, i.e.\  up to isomorphism  it
is independent of the choice of the base point.

\begin{rmk}\label{rmk:hodgegen}
 Assume  that $\mathbb L$ underlies a $\mathbb Z$-local system. The set $\Sigma$ of points $x\in X$
 with
${\rm MT}(\mathbb L_x, {\sf F}_x) \ne  {\rm MT}((\mathbb L, {\sf F}),x)$,
or equivalently with $\mathrm{Mo}(\mathbb L,x)^0\not\subset{\rm MT}(\mathbb L_x, {\sf F}_x)$,
is called the  {\it Hodge-exceptional locus}. Its complement is called the {\it
  Hodge-generic locus}. By~\cite{CDK95} and~\cite[Prop.~7.5]{Del72},   $\Sigma$ is a countable union of proper closed algebraic
subsets of $X$;  for an o-minimal proof see~\cite{BKT20}.
\end{rmk}

\subsection{Reminder on Simpson's non-abelian  Hodge  locus}  \label{sec:NL}

In this section we formulate a generalization of Simpson's results in
\cite[Sec.~12]{Sim97}.
In contrast to Simpson we formulate the result  only  ``on the Betti side'' of the non-abelian Hodge correspondence.

Let $f \colon X\to S$ be a good holomorphic map.
As recalled in Subsection~\ref{sec:Artin},
$f$ is a topological  fiber bundle, so  the sheaf of  pointed sets
$ R^1f_* \GL_r(\C) $
is a local system on $S$, with fiber $H^1(X_s, \GL_r(\C))$, the set of isomorphism classes
of rank $r$ complex local systems on $X_s$.  Its associated \'etale space  is a covering $\epsilon: T\to S$
(\cite[I.86~Prop.~9]{Bou16}). We endow $T$ with the usual complex manifold structure. As a
topological space $T$ can be constructed as
\ga{}{T=\big( \tilde{ S} \times H^1(X_{s_0}, \GL_r(\C))\big)/\pi_1(S,s_0) \notag}
where $\tilde S\to S$ is a universal covering and
the $\pi_1(S,s_0)$-action is diagonal, for a fixed base point $s_0\in S$.

 \medskip

  We write $[\mathbb L_s]\in H^1(X_s, \GL_r(\C)) $ for the isomorphism class of a local
  system $\mathbb L_s$ on $X_s$.
 Simpson defines the {\it non-abelian Hodge locus} or {\it Noether-Lefschetz locus} $\mathrm{NL}={\rm NL}(f,r)$     of rank $r$  as the  subset of $T$ consisting of those $[\mathbb L_s]$ such that
 $\mathbb L_s$ underlies a $(\Z) \C\text{-}{\sf PVHS}$ on $X_s$.

Simpson \cite[Thm.~12.1]{Sim97} proved the following theorem for $f$ proper. The extension to
non-proper $f$ is explained in Appendix~\ref{app:nonabhodge}.

\begin{thm}  \label{thm:SimpNL}
  For a good holomorphic map $f\colon X\to S$ the non-abelian Hodge locus
  ${\rm NL}(f,r)$ is a closed analytic subset of $T$ which is finite over $S$.
\end{thm}

 In fact the finiteness part of Theorem~\ref{thm:SimpNL}
is due to Deligne~\cite{Del87}.

\section{Proof of the main results}\label{sec:proof_main}

\subsection{Setting}

Let $f \colon X\to S$ be a  good algebraic morphism  satisfying {\bf Ass} from Section~\ref{sec:intro}.

In case III) the  existence of a  continuous section
 implies that $\pi_2(X,s)\to
\pi_2(S,s) $ is surjective, so by the exact sequence~\eqref{eq:longexhom} the map $ \pi_1(X_s,x)\to \pi_1(X,x)$ is injective.
From Propositions~\ref{prop:Kpi1} and~\ref{prop:hyp} it follows
that under the assumptions I) or II) the map $\pi_1(X_s,x)\to \pi_1(X,x)$ is injective as well. Furthermore, the map
of profinite completions $\widehat{\pi_1(X_s,x)}\to \widehat{ \pi_1(X,x)}$ is injective by
using Proposition~\ref{prop:ander1} in case III) and Propositions~\ref{prop:Kpi1} and~\ref{prop:hyp} in
cases I) and II).

\subsection{Proof of Theorem~\ref{introthm2} and Corollary~\ref{cor:introcanrepsurfg}}

Theorem~\ref{introthm2} is a direct consequence of Theorem~\ref{thm:maingroup} by using
the equivalence between representations of the fundamental group and local systems.

In order to deduce Corollary~\ref{cor:introcanrepsurfg} one only has to observe that semisimple
representations of surface groups  which have finite orbit under the mapping class group,
or, what is the same, finite orbit under the fundamental group of the moduli of those Riemann surfaces,  have semisimple algebraic monodromy groups.
This follows from Lemma~\ref{lem:semi-simplealg} and the fact that in rank one such  a
representation  has finite monodromy, see e.g.~\cite[Thm.~1.1]{BKMS18}.

\subsection{Proof of Theorem~\ref{thm:main}}\label{subsec:proofthmmain}

The implications 3) $\Rightarrow $ 2) $\Rightarrow $ 1), 3)  $\Rightarrow $ 4) and  4) $\Rightarrow$ 5) are clear.

We prove 1) $\Rightarrow$ 3). By Theorem~\ref{thm:ss}(v) the algebraic monodromy group of
$\mathbb L_s$ is semisimple, so we can apply Theorem~\ref{introthm2} to deduce the
existence of an algebraically isomonodromic extension $\mathbb L$ of $\mathbb L_s$, which
  can be defined over $\mathbb Z$. The polarization $Q_s$ essentially extends uniquely to
a perfect pairing $Q\colon \mathbb L \otimes \overline{\mathbb L}\to  K$. In order to show
that  the Hodge filtration $\mathsf F_s$ extends uniquely to a filtration of   $\mathbb
L\otimes_\C \mathcal O_X$ by holomorphic subbundles
$\mathsf F^i$ which make $(\mathbb L,\mathsf F,Q)$ a $K \text{-}{\sf PVHS} $ we  may assume without
loss of generality that $K=\mathbb C$. Let
\begin{equation}\label{eqmainsec:vandec}
  \mathbb L\cong \bigoplus_i \mathbb L_i \otimes_\C V_i
\end{equation}
be the canonical decomposition, where $\mathbb L_i$ are non-isomorphic simple local
systems   and where $V_i=\mathrm{Hom}(\mathbb L_i , \mathbb L)$ are complex
  vector spaces.  By Lemma~\ref{lem:algisoadmcom} we can assume that $\mathbb L_i|_{X_s}$ is simple for
all indices $i$, after replacing $S$ by a finite \'etale covering. Then by Theorem~\ref{thm:ss}(iv) there are $\C\text{-}{\sf PVHS} $
$(\mathbb L_i|_{X_s},F_{(i),s},Q_{(i),s})$ and $\C\text{-}{\sf PHS} $
$(V_i, F_{(i),V},Q_{(i),V})$ such that $(\mathbb L_s, F_s,Q_s)$ is the direct sum of their
tensor products via the isomorphism~\eqref{eqmainsec:vandec}.

With the notation of
Proposition~\ref{prop:varhodgecrit} for any $\lambda\in S^1$ we essentially obtain
that $\mathbb L_i^{\lambda}\cong \mathbb L_i $   as such an isomorphism exists after
restricting to $X_s$ and as admissible extensions are essentially unique by
Proposition~\ref{prop:uniqextadmis}. So by Proposition~\ref{prop:varhodgecrit} we see that
$\mathbb L_i$ underlies a $\C\text{-}{\sf PVHS} $ $(\mathbb L_i,F_{(i)},Q_{(i)})$ which by
Theorem~\ref{thm:ss}(ii) can be chosen to extend
$(\mathbb L_i|_{X_s},F_{(i),s},Q_{(i),s})$. We can finally endow $\mathbb L$ with the
$\C\text{-}{\sf PVHS} $ induced via the isomorphism~\eqref{eqmainsec:vandec}.

 We finally prove 5) $\Rightarrow$ 1), the  implication 4) $\Rightarrow$ 1) is
analogous and we omit the details.  Using Artin approximation~\cite{Art69} we can approximate the formal filtration
$\hat{\mathsf F}$ by a formal filtration with the same fiber over $s$ but which extends to a (relative) Griffiths
transversal filtration by subbundles $\mathsf F$ of
the isomonodromic extension
\[
  (\mathcal E_\Delta, \nabla\colon \mathcal E_\Delta \to
  \Omega^1_{X_\Delta/\Delta}( \mathcal E_\Delta  ) ) \]
over $X_\Delta$, where $\Delta\subset S$ is a suitable contractible open neighborhood of
$s$. Let $\mathbb  L_\Delta$ be the extension of the local system $\mathbb  L$ to
$X_\Delta$, let similarly $Q_\Delta$ denote the extension of the polarization.
As in Appendix~\ref{app:extreps}, see Claim~\ref{app1:claimopen}, the set of $t\in \Delta$ such that $(\mathbb
L_\Delta|_{X_t} , \mathsf
F|_{X_t}, Q_\Delta|_{X_t} )$ is a $\C\text{-}{\sf PVHS} $ is open in $\Delta$,
 note that
in  the non-proper case this observation relies on~\cite[Prop.~7.1.1]{KU09}.  By the same reasoning as in
the proof of Theorem~\ref{thm:del_fixpart}    (5) $ \Rightarrow ( 1)$, the connected component of the non-abelian
Hodge locus $\mathrm{NL}(f,r)$ containing $[\mathbb L_s ]$, see Subsection~\ref{sec:NL}, is
a finite \'etale covering of $S$. So the monodromy orbit of $[\mathbb L_s]$ is finite.

\subsection{Proof of Remark~\ref{rmk:introMT}}

By the argument in Subsection~\ref{subsec:proofthmmain} one can find an extension of variations of
Hodge structure as in Theorem~\ref{thm:main}(3) which is algebraically isomonodromic, i.e.\
such that
\[\mathrm{Mo}(\mathbb L_s,x)^0\xrightarrow\sim \mathrm{Mo}(\mathbb L,x)^0\]
is an
isomorphism. Then by Subsection~\ref{ss:MT} we deduce an isomorphism of Mumford-Tate groups of
variations
\[
  \mathrm{MT}((\mathbb L_s,F_s),x)\xrightarrow\sim \mathrm{MT}((\mathbb L,F),x).
\]
This shows that the Hodge-generic locus of  $(\mathbb L,F)$ intersected with $X_s$
coincides with the Hodge-generic locus of  $(\mathbb L_s,F_s)$. But the latter is
non-empty by Remark~\ref{rmk:hodgegen}.  This finishes the proof.

\appendix

\section{Non-abelian Hodge loci  (after Simpson)}\label{app:nonabhodge}

In this appendix we explain how to generalize Simpson's proof of
Theorem~\ref{thm:SimpNL} from the proper case to good holomorphic maps $f$.  We assume in the following that $f\colon X\to S$ is good holomorphic, where $S$ can
be an arbitrary complex manifold. As the statement is local on $S$,  we can fix
$s_0\in S$ and allow ourselves to replace $S$ by an open neighborhood of $s_0$. We
assume without loss of generality that $S$ is  contractible.

Simpson observes that Deligne's proof of his finiteness theorem~\cite{Del87} immediately
generalizes to the following proposition.

\begin{prop}\label{app1:prop_fini}
After possibly shrinking $S$ around $s_0$ there are only finitely many  leaves of the (trivial) fibration $T\to
S$ which meet $\mathrm{NL}(f,r)$.
\end{prop}

In view of Proposition~\ref{app1:prop_fini}
we can  fix a  $\mathbb C$-local
system $\mathbb L$ on $X$ of rank $r$ which underlies a $\Z$-local system. Then to prove  Theorem~\ref{thm:SimpNL} it is sufficient to show that the set
of $s\in S$ such that
$\mathbb L_s$ underlies a $\C \text{-} {\sf PVHS}$ is closed analytic in $S$.

\smallskip

{\it A reduction}.

We reduce to  $f$ of relative dimension one and $\mathbb L$ unipotent along $\overline
X\setminus X$. The first condition  can be achieved by choosing a sufficiently  generic
relative hyperplane section    $Y\hookrightarrow X$ after shrinking $S$ around $s_0$. By
Bertini, $Y\to S$ is then good holomorphic and Remark~\ref{intro:rmk1} tells us that for  $\dim(X_s)>1$
\[
\mathbb L_s \text{ underlies a } \C  \text{-} {\sf PVHS}\quad  \Leftrightarrow\quad \mathbb L_s|_Y \text{ underlies a } \C  \text{-} {\sf PVHS}
\]
for any $s\in S$.  After performing this Bertini argument successively we can assume without
loss of generality that $f$ is of relative dimension one.

By Theorem~\ref{thm:ss}(v) the monodromy of $\mathbb L$ along  $\overline X\setminus X$ is
quasi-unipotent.
By pullback along a finite ramified covering of $\overline X$ we can reduce to the case in
which
this monodromy is unipotent and then conclude for the original $X$ using Proposition~\ref{prop:qfinhodgevar}.
This argument also allows us to assume that $\overline X\setminus X$
consists of images of finitely many disjoint sections of $\overline X\to S$.

\smallskip

{\it The flag variety}.

By Theorem~\ref{thm:ss}(iv) we have to show that for each simple constituent
$\mathbb M$ of $\mathbb L$ the set of $s\in S$ such that $\mathbb M_s$
underlies a  $\C \text{-} {\sf PVHS}$ is closed analytic in $S$. We will see that this
holds as a consequence of the key Proposition~\ref{prop:simprid} below.

We fix such a simple constituent $\mathbb M$ and a hermitian perfect
pairing $Q\colon \mathbb M\times \overline{\mathbb M}\to \C$, which is
unique up to a factor in $\mathbb R^\times$. Let $\mathcal E$ be the Deligne extension of
$\mathbb M\otimes_{\C} \sO_X$ to a holomorphic bundle on $\overline X$.

Grothendieck constructed a certain flag scheme, see~\cite[2.A.1]{HL10}, which for us is a morphism of
analytic spaces $\psi\colon G\to S$, where a point $\xi\in G$ corresponds to a descending
finite filtration by coherent subsheaves $\mathsf F^i_\xi$ ($i\ge 0$) of
$\mathcal E_{\psi(\xi)}$ such that $\mathcal E_{\psi(\xi)}/\mathsf F^i_\xi$ is torsion
free for all $i$ and such that Griffiths transversality
$\nabla \mathsf F^i \subset \Omega_{X_{\psi(\xi)}}^1(\mathsf F^{i-1})$ holds for all $i\in
\mathbb Z$. In order to
get rid of the ambiguity of shifting the filtration, we also assume that we
consider only such filtrations with $\mathsf F^0_\xi= \mathcal E_{\psi(\xi)}$ and
$\mathsf F^1_\xi\ne \mathcal E_{\psi(\xi)}$. Let $G^0$ be the subset of $\xi\in G$
corresponding to the filtrations by subbundles such that $\mathsf F_\xi|_{X_{{\psi(\xi)}}}$ is a
variation of Hodge structure polarized by $Q$.

\begin{prop}\label{prop:simprid} The following holds:
  \begin{itemize}
  \item[1)] $G$ is the countable disjoint union of analytic spaces which are projective
    over $S$, more precisely the  points  $\xi\in G$ such that the ranks and degrees  of all
    coherent sheaves $\gr_{\mathsf F}^i
    \mathcal E_{\psi(\xi)}$ are fixed form a clopen subset of $G$ which is  projective over $S$;
\item[2)] $G^0$ consists of finitely many connected components of $G$.
  \end{itemize}
\end{prop}

Proposition~\ref{prop:simprid} and Remmert's proper mapping theorem imply that $\psi(G^0)$
is a closed analytic subset of $S$. On the other hand by analytic
continuation of the
Hodge filtration along the Deligne extension, i.e.\ the nilpotent orbit theorem~\cite[(2.1)]{CK89},  this is precisely the above set of $s\in S$
such that $\mathbb M_s$
underlies a  $\C \text{-} {\sf PVHS}$ polarized by $Q$.

\begin{proof}[Proof of Proposition~\ref{prop:simprid}]
  Part 1) is immediate from Grothendieck's results, see~\cite[2.A.1]{HL10}.

  \smallskip

  For part 2) we have to show that there are only finitely many  collections of numbers
  which can occur as the ranks and degrees as in 1) for elements of $ G^0$. The finiteness
  of the number of ranks follows from the fact that   we assumed $ \mathsf F_\xi^0\not\subset
  \mathsf F^1_\xi\ne \mathcal E_{\psi(\xi)}$ and that
  there can be no gap in the Hodge
  filtration of a simple local system by Theorem~\ref{thm:ss}(ii), i.e.\ if the rank of
  $\gr_{\mathsf F}^i \mathcal E_{\psi(\xi)}$ vanishes for some $i>0$ then it also vanishes
  for all larger  indices $i$.   Similarly, the degrees are bounded as the associated log
  Higgs bundle is  stable \cite[Thm.5]{Sim90}  and we can therefore apply \cite[Prop.~3.2]{Nit91}
  to $L_s=\omega_{\bar X_s}(\bar X_s\setminus X_s), \ s\in S $  which  has a constant degree for all fibers $X_s$ of $X \to S$.

  To conclude the proof of part 2) we have to show that $G^0$ is both open and closed in
  $G$, which is a consequence of Claim~\ref{app1:claimopen} and Claim~\ref{app1:claimharm}.
\end{proof}

\begin{claim}\label{app1:claimopen}
  The subset $G^0$ of  $G$ is open.
\end{claim}

\begin{proof}[Proof of Claim~\ref{app1:claimopen}]
Let us consider the subset $Z\subset G \times_S \overline X$ consisting of those points $(\xi,x)$
such that either $x\in X$ and $\mathsf F_{\xi,x}$ is a Hodge filtration polarized by $Q$
or  $x\in \overline X\setminus X$ and $\mathsf F_{\xi,x}$ together with the nilpotent
monodromy operator and $Q$ is a nilpotent orbit. We have to show that $Z$ is an open
subset.  Indeed,   this implies that  $G^0$ is open as it   is then equal to the complement of the image of the closed subset $ (G \times_S \overline X)\setminus Z$
along the proper projection $G\times_S \overline X\to G$.

The subset $Z$ is open around  $G \times_S X$, because the Hodge-Riemann inequality is an
open condition. Around $G\times_S (\overline X\setminus X)$ the subset $Z$ is open as the
property of being  a nilpotent orbit, granted Griffiths transversality holds, is an open condition by
\cite[Prop.\ 7.1.1]{KU09}  and  as a Griffiths transversal filtration by subbundles, which can be approximated by a
nilpotent orbit, is a variation of Hodge structure locally by Remark~\ref{rmkdeghodgunifo}.
\end{proof}

\begin{claim}\label{app1:claimharm}
The subset $G^\circ$ of $G$ is closed.
\end{claim}

In the proof of Claim~\ref{app1:claimharm} we need the following continuity result,
Proposition~\ref{prop:fibreharmonic}, for a fiberwise
harmonic metric due to T.~Mochizuki~\cite[Prop.~4.2]{Moc09}.
We note that $\det \mathbb M$ is finite by Theorem~\ref{thm:ss}(v), so we can fix a positive flat hermitian
metric $h^{\det}$ on  $\det \mathbb M$, which is automatically pluri-harmonic, to be
concrete take $h^{\det}=\pm Q^{\det}$.

\begin{prop}\label{prop:fibreharmonic}
Let $h\colon  \mathcal E|_X \times  \overline{  \mathcal E }|_X\to  \mathbb C$ be the unique fiberwise
harmonic, tame and
purely imaginary metric with determinant $h^{\det}$. Then $h$ is continuous.
\end{prop}

\begin{proof}[Proof of Claim~\ref{app1:claimharm}]
  Consider a sequence $\xi(n)\in G^\circ$ converging to $\xi \in G$. In order to show that
  $\xi \in G^0$ we have to  verify two properties: 1) $\mathsf F_\xi^i$ is a subbundle of
  $\mathcal E_{\psi(\xi)}$ for all $i$, 2) the $ \mathsf F_\xi|_{X_{\psi(\xi)}}$ form a
  variation of Hodge structure polarized by $Q$.

  Consider a point $x\in X_{\psi(\xi)}$ around which 1) holds; note that a priori there might be
  finitely many points for which it fails. The orthogonal
  complement $(\mathsf F^i_{\xi,x})^\perp$ of $\mathsf F^i_{\xi,x}$ in $\mathbb M_x$ with respect to
  $Q_x$ and with respect to the fiberwise harmonic metric $h_x$ from
  Proposition~\ref{prop:fibreharmonic} coincide by continuity since they coincide on
  $\mathcal E_{\psi(\xi(n))}^i$ for all $n$ by the description of the metric $h$ as a Hodge metric in terms of $Q$
  on the fiber of $X$ over $\psi(\xi(n))$. As $h$ is positive definite we deduce that
  $\mathbb M_x = \mathsf F^i_{\xi,x} \oplus (\mathsf F^i_{\xi,x})^\perp$. In other words $Q_x$
  restricted to $\mathsf F^i_{\xi,x}$ is a perfect pairing. For the same reason $(-1)^i Q_x$
  restricted to $\mathsf F^i_{\xi,x} \cap (\mathsf F^{i+1}_{\xi,x})^\perp$ is positive definite as by
  continuity it agrees with $h_x$ there. So around $x$ we get a
  variation of Hodge structure $\mathsf F_{\xi}$.

  In order to conclude we have
  to show that the torsion vanishes in the coherent sheaf $\mathcal E_{\psi(\xi)}/\mathsf
  F^i_{\xi,\psi(\xi)}$ for any $i$. Let $\overline \xi  (n)$ correspond to the orthogonal Hodge
  filtration $\overline{\mathsf F}^{a}= ({\mathsf F}^{-a})^\perp$ on $X_{\psi(\xi(n))}$
  continuously extended to the anti-holomorphic Deligne extension $\overline{\mathcal E}_{\psi(\xi(n))}$
  on $\overline X_{\psi(\xi(n))}$. By the properness of $G$ over $S$,
  Proposition~\ref{prop:simprid}~1), we can replace $(\xi(n))_n$ by a subsequence such
  that $\overline \xi  (n)$ converges to some $\overline{\xi}$ in the anti-holomorphic version of $G$.
By the continuity of the degree on $G$ and by Lemma~\ref{lem:degsum} we obtain
  \[
    \deg( \mathsf F^i_{\xi,\psi(\xi)}) + \deg (  \mathsf{F}^{-i}_{\overline{\xi},\psi(\xi)}) = 0
  \]
for all $i$ as this holds for the Hodge filtrations corresponding to $\xi(n)$ for all $n$.
  By  Schmid's nilpotent orbit theorem, see~\cite[(2.1)]{CK89}, the
  saturations $\tilde{\mathsf F}_{\xi,\psi(\xi)}$ of $\mathsf F_{\xi,\psi(\xi)}$ and
  $\tilde{{\mathsf  F}}_{\overline \xi,\psi(\xi)}$ of ${\mathsf F}_{\overline \xi,\psi(\xi)}$  form a
  variation of Hodge structure on $X_{\psi(\xi)}$ of weight $-1$, so
    \[
  \deg( \tilde{\mathsf F}^i_{\xi,\psi(\xi)}) + \deg ( \tilde{ {\mathsf
      F}}^{-i}_{\overline \xi,\psi(\xi)}) =0
\]
by Lemma~\ref{lem:degsum}.
Therefore,  the inequalities
\[
\deg( \mathsf F^i_{\xi,\psi(\xi)}) \le  \deg( \tilde{\mathsf F}^i_{\xi,\psi(\xi)}) ,
\quad \deg( {\mathsf F}^{-i}_{\overline \xi,\psi(\xi)}) \le  \deg( \tilde{{\mathsf
    F}}^{-i}_{\overline \xi,\psi(\xi)})
\]
have to be equalities and consequently $ \mathsf F^i_{\xi,\psi(\xi)} =  \tilde{\mathsf F}^i_{\xi,\psi(\xi)}$.
\end{proof}

\section{Extensions of  representations of profinite groups (after Simpson)}\label{app:extreps}

Let $E$ be a finite field extension of $\Q_\ell$. Let $H\subset G$ be a closed normal
subgroup of a profinite group $G$. Set $\Gamma= G/H$.   Let $\rho_H\colon H\to \mathrm{GL}_r(E)$
be a continuous representation and denote by $[\rho_H]$ the isomorphism class of $\rho_H$.
Then $\Gamma$ acts on the set of  isomorphism classes of continuous
representations $H\to \mathrm{GL}_r(E)$ by conjugation.

\begin{lem}
If $\rho_H$ is semisimple then the stabilizer group $\Gamma_{[\rho_H]}$ is closed in $\Gamma$.
\end{lem}

For a proof of the lemma see for example~\cite[Sect.~5]{Zoc24} in which the author endows
the space of isomorphism classes of   $\ell$-adic representations of rank $r$ with a
uniform $\ell$-adic topology on which the action of $\Gamma$ is continuous.

\smallskip

We say that a property related to $G$ holds {\it essentially} if it holds after replacing
$G$ by an open subgroup $G'$ with $H\subset G'$.

\begin{prop}\label{prop:appext}
  If $\rho_H$ is absolutely simple, $\det(\rho_H)$  is finite and   $\Gamma$ fixes $
  [\rho_H]$,  there essentially exists a unique extension of $\rho_H$ to a
continuous representation $ \rho\colon G\to  \mathrm{GL}_r(E)$ with finite determinant.
\end{prop}

This result originates in the work of Simpson \cite[Proof of Thm.~4]{Sim92}. Variants
of it can be found in \cite[Proof of Prop.~3.1]{EG18} and
\cite[Prop.~3.1.1]{Lit21}. We are not aware of a reference for
Proposition~\ref{prop:appext} in the literature, so we sketch an argument the details of
which  can
be found in Hugo Zock's master thesis~\cite{Zoc24}.

\begin{proof}[Proof sketch]
For $M\in\mathfrak{gl}_n(E)$ and $P =[W] \in \mathrm{PGL}_r(E)$ we denote the conjugation
  action by $  ^PM = W M W^{-1}$.
For each $g\in G$ there exists a unique $P_g\in \mathrm{PGL}_r(E)$ with $^{P_g}\rho_H(h)=
\rho_H(g h g^{-1})$. We claim that the map $g\mapsto P_g$ is continuous. In order to check this we
have to show that for any $M\in\mathfrak{gl}_n(E)$,  the map $g\mapsto  {}^{P_g} M$ is
continuous. As $\rho_H$ is absolutely simple,  every such matrix $M$ is an $E$-linear
combination of matrices $\rho_H(h)$ for varying $h\in H$, so it is sufficient to observe that for any $h\in H$
the map
\[
g\mapsto  {}^{P_g} \rho_H(h) = \rho_H(g h g^{-1})
\]
is continuous.
So we see that we obtain a continuous  extension
\[
  \mathbb P\rho\colon G\to \mathrm{PGL}_r(E)\quad \text{of}\quad
  \mathbb P \rho_H \colon H \to \mathrm{PGL}_r(E).
\]
In order to prove essential uniqueness in Proposition~\ref{prop:appext} let $\rho $ and
$\tilde \rho$ be two lifts of $\mathbb P\rho$ to two
continuous representations of $G$  with
finite determinants and which agree on $H$. Then $\tau(g) = \rho(g) \tilde \rho(g)^{-1}\in K^\times$ defines a continuous
homomorphism $\tau\colon \Gamma\to K^\times$ with finite image. We deduce $\rho = \tilde \rho$
after restricting to the kernel of $\tau$.  This proves the uniqueness part.  See Proposition~\ref{prop:uniqextadmis} for a similar method.

For  the  existence  part  in Proposition~\ref{prop:appext}  we first check that it is sufficient to construct $\rho$ with values in
$\GL_r(\tilde E)$, where $\tilde E$ is a finite Galois extension of $E$. Indeed,
then for any $\eta\in \mathrm{Gal}(\tilde E/E)$ the extensions $\eta\cdot \rho$ and $\rho$ are
essentially equal by the uniqueness part.
So,  as $ \mathrm{Gal}(\tilde E/E)$ is finite, we can
 replace $G$ by an open subgroup containing
$H$ so they  are equal for all $\eta$, i.e.\ $\rho$ takes values in $\GL_r(E)$.

Consider a finite extension $\tilde E$ of $E$
 in which every element of $E^\times$ is an $r$-power of an element of $\tilde E^\times$.
Replace $E$ by $\tilde E$. Then $\mathbb P\rho$  has values  in the closed subgroup $\mathrm{PSL}_r(E)$
of $ \mathrm{PGL}_r(E)$.
Consider the strict exact sequence of topological groups
\[
1\to \mu_{r}(E) \to \mathrm{SL}_r(E)\to \mathrm{PSL}_r(E)\to 1
\]
in which the surjection on the right has a continuous splitting in topological spaces. So
the standard cocycle argument, see~\cite[App.~B]{Zoc24}, gives an exact sequence
\[
H^1_{\rm cont}(G,\mathrm{SL}_r(E) )\to H^1_{\rm cont}(G,\mathrm{PSL}_r(E) )
\xrightarrow{\rm Ob} H^2_{\rm cont}(G,\mu_r(E)).
\]
As the restriction of $\mathrm{Ob}(\mathbb P \rho)$ to $H$ vanishes, the Hochschild-Serre
spectral sequence  implies that ${\rm Ob} ( \P\rho)$ lies in a subgroup of
$H^2_{\rm cont}(G, \mu_r(E))$  involving
$H^a(\Gamma, H^{2-a}(H, \mu_r(E)))$ for $a=1,2$ only. Those two groups die on a finite index subgroup of $\Gamma$.
So we can assume  that $\mathrm{Ob}(\mathbb P \rho) =0$.

So we can lift $\mathbb P
\rho$ to a
continuous representation $\check \rho\colon G\to \mathrm{SL}_r(E)$.
Now \[
 \tau\colon H\to E^\times ,\quad  \tau(h) = \rho_H(h) \check\rho(h)^{-1}\in E^\times
\]
is a continuous homomorphism with values in $ \mu_N(E)$
for some $N>0$ by the condition on finite determinants. Note that $\tau$  is fixed by the $G$-action by conjugation. Using the exact
inflation-restriction sequence
\[
H^1_{\rm cont}(G,\mu_N(E)) \to H^1_{\rm cont}(H,\mu_N(E))^G \to H^2_{\rm cont}(\Gamma, \mu_N(E))
\]
we see that after replacing again $G$ by an open subgroup $G'$ with $H\subset G'$ we can lift $\tau$
to a continuous homomorphism $\check\tau\colon G\to \mu_N(E)$. Finally, we deduce that $\rho=
\check \tau \check \rho\colon G \to \GL_r(E)$ is the requested continuous extension of $\rho_H$.
\end{proof}

\end{document}